\documentclass[a4paper,11pt]{amsart}

\usepackage{amsmath}
\usepackage{amssymb}
\usepackage{amsthm}
\usepackage[abbrev,alphabetic]{amsrefs}
\usepackage{amscd}

\usepackage{amsthm}

\usepackage[all]{xy}
\usepackage{xypic}

\title[Rational points on log Fano $3$-folds]
{Rational points on log Fano threefolds 
over a finite field}

\author{Yoshinori  Gongyo}

\address{Graduate School of Mathematical Sciences, 
the University of Tokyo, 3-8-1 Komaba, Meguro-ku, Tokyo 153-8914, Japan.}

\email{gongyo@ms.u-tokyo.ac.jp}

\author{Yusuke Nakamura}

\address{Graduate School of Mathematical Sciences, 
the University of Tokyo, 3-8-1 Komaba, Meguro-ku, Tokyo 153-8914, Japan.}

\email{nakamura@ms.u-tokyo.ac.jp}

\author{Hiromu Tanaka}
\address{Department of Mathematics, Imperial College, London, 180 Queen's Gate, 
London SW7 2AZ, UK} 
\email{h.tanaka@imperial.ac.uk}

\subjclass[2010]{Primary 14J45; Secondary 14G05, 14E30}

\keywords{rational points, Fano varieties, Witt rationality, rational chain connectedness}

\newtheorem{thm}{Theorem}[section]
\newtheorem{lem}[thm]{Lemma}

\newtheorem{cor}[thm]{Corollary}
\newtheorem{prop}[thm]{Proposition}

\newtheorem{claim}[thm]{Claim}

\theoremstyle{definition}
\newtheorem{defi}[thm]{Definition}
\newtheorem{example}[thm]{Example}
\newtheorem{conj}[thm]{Conjecture}

\theoremstyle{remark}
\newtheorem{rmk}[thm]{Remark}
\newtheorem*{ackn}{Acknowledgements}

\newcommand{\MO}{\mathcal{O}}

\newcommand{\mbN}{\mathbb{N}}
\newcommand{\mbQ}{\mathbb{Q}}
\newcommand{\Q}{\mathbb{Q}}
\newcommand{\mbF}{\mathbb{F}}

\newcommand{\mbZ}{\mathbb{Z}}

\newcommand{\mcO}{\mathcal{O}}

\newcommand{\red}[0]{{\operatorname{red}}}
\newcommand{\Ex}[0]{{\operatorname{Ex}}}

\DeclareMathOperator{\Supp}{Supp}

\DeclareMathOperator{\Spec}{Spec}

\DeclareMathOperator{\Zeros}{Zeros}

\DeclareMathOperator{\coeff}{coeff}

\DeclareMathOperator{\ob}{ob}

\begin{document}
\begin{abstract}
We prove the $W\mathcal{O}$-rationality of klt threefolds 
and the rational chain connectedness of klt Fano threefolds 
over a perfect field of characteristic $p>5$. 
As a consequence, any klt Fano threefold over a finite field has a rational point. 
\end{abstract}

\maketitle

\tableofcontents

\section{Introduction}
The problem whether a certain system of polynomials has a common solution or not is 
classical but has been of great interest in the history of mathematics. 
In the geometrical context, 
Fano varieties defined over a certain field $k$ are believed to have a $k$-rational point. 
For instance, the classically known Chevalley--Warning theorem states that 
homogeneous polynomials $f_1, \ldots, f_l \in \mbF _q[x_0, \ldots, x_n]$ with $(n+1)$-many variables 
over a finite field $\mbF _q$ have a non-trivial common solution if $\sum _{1 \le i \le l} \deg f_i \le n$ holds. 
More precisely, it states that the number of common solution is divisible by the characteristic $p$ of 
the field $\mbF _q$. 
Geometrically, it can be interpreted as 
that the number of the $\mbF _q$-rational points on 
a complete intersection Fano variety is one modulo $p$. 
This kind of problem was answered in the smooth case by Esnault \cite{Esnault}: 
\begin{thm}[{Esnault \cite{Esnault}}]\label{thm:Esnault}
Let $X$ be a smooth geometrically connected projective variety over a finite field $\mbF _q$ with $q$ elements. 
If $-K_X$ is ample, then 
the number of $\mbF _q$-rational points satisfies $\# X (\mbF _q) \equiv 1 \pmod q$. 
\end{thm}
\noindent
One motivation of this paper is to generalize this result to three dimensional varieties with singularities. 
Indeed we prove that the same statement holds for three dimensional varieties of Fano type 
defined over a finite field of characteristic larger than five. 
\begin{thm}\label{mainthm:ratpo}
Let $X$ be a three dimensional normal geometrically connected projective variety over a finite field 
$\mbF _q$ of characteristic larger than five. 
If $X$ is of Fano type, 
then the number of $\mbF _q$-rational points satisfies $\# X (\mbF _q) \equiv 1 \pmod q$. 
\end{thm}
\noindent
Here, we say that a normal projective variety $X$ defined over a field 
is \textit{of Fano type} if there exists an effective $\mbQ$-divisor $\Delta$ such that 
$(X, \Delta)$ is klt (see Subsection \ref{section:log_pair} for the definition) and $-(K_X + \Delta)$ is ample. 

By \cite[Theorem 1.1]{BBE} and the Lefschetz trace formula, 
Theorem \ref{mainthm:ratpo} is a consequence of the following proposition, 
which states the vanishing of the Witt vector cohomology 
(see Subsection \ref{subsection:witt} for the definition). 

\begin{thm}\label{mainthm:WV}
Let $X$ be a three dimensional normal projective variety over a perfect field 
of characteristic larger than five. 
If $X$ is of Fano type, then $H^i (X, W\mcO_{X, \mbQ}) = 0$ holds for $i>0$. 
\end{thm}

\noindent
It can be seen as an analogy of the vanishing of the structure sheaf $H^i (X, \mcO_{X}) = 0 \ (i>0)$, 
which is true in characteristic zero for varieties of Fano type of arbitrary dimension by the Kawamata--Viehweg vanishing theorem. 
In the case of positive characteristic, 
we do not know whether the vanishing $H^i (X, \mcO_{X}) = 0 \ (i>0)$ holds or not 
by the lack of the Kodaira vanishing theorem, 
although the vanishing $H^i (X, \mcO_{X}) = 0 \ (i>0)$ implies 
the vanishing of $H^i (X, W\mcO_{X, \mbQ}) = 0 \ (i>0)$ (cf.\ Lemma \ref{usual-Witt-higher}). 
However, by \cite{BE}, the vanishing in Theorem \ref{mainthm:WV} is known to be true 
for a variety $X$ with the following two properties: 
\begin{itemize}
\item $X$ has $W\mcO$-rational singularities, and 
\item $X$ is rationally chain connected. 
\end{itemize}
Hence, this paper is devoted to prove the following two theorems.

\begin{thm}[{$=$\ Theorem \ref{thm:WO}}]\label{mainthm:WR}
Let $(X, \Delta)$ be a three dimensional klt pair over a perfect field in characteristic larger than five. 
Then $X$ has $W \mcO$-rational singularities, that is, 
there exists a proper birational morphism $\varphi:V \to X$ from a smooth threefold $V$ 
such that $R^i\varphi_*(W\MO_{V, \Q})=0$ holds for $i>0$. 
\end{thm}

\begin{thm}[{$=$\ Proposition \ref{thm:rcc_Fanotype}}]\label{mainthm:rcc}
Let $X$ be a three dimensional projective variety of Fano type over a perfect field 
of characteristic larger than five. 
Then $X$ is rationally chain connected. 
\end{thm}

\begin{rmk}\label{rmk:proof}
In Section \ref{section-rational-pt}, 
we shall prove Theorem \ref{mainthm:ratpo} directly from 
Theorem \ref{mainthm:WR} and Theorem \ref{mainthm:rcc} using the argument in \cite{Esnault} and \cite{BBE}. 
\end{rmk}

Theorem \ref{mainthm:WR} can be seen as an analogy of the following fact in characteristic zero 
which is proved in \cite[Section 1-3]{KMM1}
(cf.\ \cite[Theorem 5.2]{KMbook}):
\begin{itemize}
\item klt varieties have rational singularities, that is, 
there exists a proper birational morphism $\varphi:V \to X$ from a smooth variety $V$ 
such that $R^i\varphi_*(\mcO_{V})=0$ holds for $i>0$. 
\end{itemize}
It is known to fail in positive characteristic. 
In Subsection \ref{subsection:non-rational_klt}, we shall see examples in characteristic two 
which are klt but with non-rational singularities. 
This example (Proposition \ref{ex:dim4}) is four-dimensional. 
However, in dimension three, 
we do not know whether klt singularities are always rational singularities or not.

On the other hand, related to Theorem \ref{mainthm:rcc}, 
Campana and Koll\'ar--Miyaoka--Mori showed in the 1990's that
smooth Fano varieties over characteristic zero are rationally connected \cite{KMM2, campana}.
Recently, Zhang and Hacon--$\mathrm{M^{c}}$Kernan generalized this to log Fano varieties 
over characteristic zero \cite{Zhang, HM}.
Then, it is natural to ask the following question:
are smooth Fano varieties over a positive characteristic field rationally connected?
This is still an open problem, but it is known
that they are rationally chain connected (cf.\ \cite[Ch V, 2.14]{Kollar2}).
The first and third authors proved the above theorem for globally $F$-regular threefolds 
with Z.\ Li, Z.\ Patakfalvi, K.\ Schwede and R.\ Zong in \cite{6authors}
(see \cite[Definition 3.1]{ScSm} for the definition of globally $F$-regular varieties). 
In this paper, we generalize this result to varieties of Fano type by the similar strategy to \cite{6authors}. 
After finishing the work \cite{6authors}, we noticed 
that Shokurov in \cite{Shokurov} claimed that this problem can be reduced to the standard conjecture 
in the minimal model theory and the resolution of singularities. 
His argument looks similar to ours but we could not get the proof in positive characteristic. 
Hence, also for the reader's convenience, we include the proof of the rationally chain connectedness. 

The paper is organized as follows:
in Section \ref{section:pre}, we review some results on the minimal model theory in positive characteristic. 
Some results are generalized over a perfect field in order to apply them to finite fields in the latter sections. 
Further in Subsection \ref{subsection:witt} and \ref{Sub-fundamental-Witt}, 
we review the definition of the Witt vector cohomology and study its properties. 
In Section \ref{S-vanishing}, we prove the $W\mcO$-rationality of klt threefolds in characteristic $p > 5$
(Theorem \ref{mainthm:WR}). Further, as a consequence in characteristic $p \leq 5$, 
we reduce the problem on the existence of flips to a certain special case (Theorem \ref{thm:235}). 
It suggests that the Witt vector cohomology is helpful also for the minimal model theory itself. 
In Section \ref{section:rcc}, we prove the rational chain connectedness of threefolds of Fano type 
in characteristic $p > 5$ (Theorem \ref{mainthm:rcc}). 
We also treat the relative setting (Theorem \ref{thm:rcc_rel}). 
In Section \ref{section-rational-pt}, 
we prove the main theorem (Theorem \ref{mainthm:ratpo}) and the vanishing theorem (Theorem \ref{mainthm:WV}). 

After we put the first draft of this paper on arXiv, Wang also put a preprint \cite{Wang} there, 
which is on the rational chain connectedness. 
We note here that Theorem A and B in \cite{Wang} follow from Theorem \ref{thm:rcc_rel} in our paper. 

\begin{ackn} 
We would like to thank Professors Yujiro Kawamata and 
Shunsuke Takagi for warm encouragements and giving advices. 	
We would also like to thank Professors Paolo Cascini and H{\'e}l{\`e}ne Esnault for the discussions. 
The first author also wish to thank Professors Shigefumi Mori and Caucher Birkar for discussing 
the rationality of the bases of conic bundles and the 
rational chain connectedness of conic bundles. 
The first and third authors wish to express their deep gratitude to the collaborators of \cite{6authors}: 
Zhiyuan Li, Zsolt Patakfalvi, Karl Schwede and Runhong Zong for many discussions during 
working our collaboration. 
We also appreciate Professors Jean-Louis Colliot-Th\'el\`ene and J\'anos Koll\'ar 
for discussing the definition of the rational chain connectedness. 
We appreciate the referees for careful reading this paper and pointing out some mistakes. 
In particular, one of the referees pointed out the difference of the definition of Witt-rationality 
between the one in \cite{CR2} and the one in the previous version of this paper (cf Remark \ref{cr-remark}). 

The first author is partially supported by the Grant-in-Aid for Scientific Research
(KAKENHI No.\ 26707002).
The second author is partially supported by the Grant-in-Aid for Scientific Research
(KAKENHI No.\ 25-3003).
\end{ackn}

\section{Preliminaries}\label{section:pre}

\subsection{Notation and convention}
We summarize notation of this paper. 

\begin{itemize}
\item Throughout this paper, 
we work over a perfect field $k$ of characteristic $p>0$ 
unless otherwise specified. 

\item We often identify an invertible sheaf $L$ with the divisor corresponding to $L$. 
For example, we use the additive notation $L + A$ for a line bundle $L$ and a divisor $A$. 

\item For a coherent sheaf $F$ and a Cartier divisor $L$, 
we define $F(L):=F\otimes \mathcal O_X(L)$. 

\item We say that $X$ is a {\em variety} over $k$ if 
$X$ is an integral scheme which is separated and of finite type over $k$. 
We say that $X$ is a {\em curve} (resp. a {\em surface}, a {\em threefold}) 
if $X$ is a variety over $k$ with $\dim X=1$ (resp. $\dim X=2$, $\dim X=3$). 
\end{itemize}

\subsection{Log pairs}\label{section:log_pair}
We recall some notation in the theory of singularities in the minimal model program. 
For more details, we refer the reader \cite[Section~2.3]{KMbook} and \cite{Kollar}. 

We say that $(X, \Delta)$ is a \textit{log pair} 
if $X$ is a normal variety and $\Delta$ is an effective $\mbQ$-divisor such that 
$K_X + \Delta$ is $\mbQ$-Cartier. 
For a proper birational morphism $f: X' \to X$ from a normal variety $X'$ 
and a prime divisor $E$ on $X'$, the \textit{discrepancy} of $(X, \Delta)$ 
at $E$ is defined as
\[
a_E (X, \Delta) := \coeff _E (K_{X'} - f^* (K_X + \Delta)). 
\]

We say that a log pair $(X, \Delta)$ is \textit{log canonical} 
if $a_E (X, \Delta) \ge -1$ for any divisor $E$ over $X$. 
Further, we say that a log pair $(X, \Delta)$ is \textit{Kawamata log terminal}, \textit{klt} for short 
(resp.\ \textit{purely log terminal}, \textit{plt} for short), if $a_E (X, \Delta) > -1$ for any divisor 
(resp.\ exceptional divisor) $E$ over $X$. 

We call a log pair $(X, \Delta)$ 
\textit{divisorial log terminal} (\textit{dlt} for short) 
when there exists a log resolution $f : Y \to X$ such that 
$a_E(X, \Delta) > -1$ for any $f$-exceptional divisor $E$ on $Y$.

We recall here the definition of log resolution. 
Let $D$ be a closed subset of a smooth scheme $X$ (over a perfect field $k$) and 
let $D_1, \ldots , D_n$ be the irreducible components of $D$ with the reduced scheme structures. 
We say that $D$ is \textit{simple normal crossing} if the scheme-theoretic intersection $\bigcap_{j \in J} D_j$  
is smooth over $k$ for every subset $J \subset \{1, \ldots , n \}$ (cf.\ \cite[Subsection 2.1]{KEW}). 
For a variety $X$ and a closed subset $Z$ of $X$, 
we say that $f:Y \to X$ is a \textit{log resolution} of $(X, Z)$ when 
$f$ is a proper birational morphism from a smooth variety $Y$ such that 
$f^{-1}(Z)$ is purely codimension one and 
$\Ex (f) \cup f^{-1}(Z)$ is simple normal crossing. In dimension three, 
there exists a log resolution for such a pair $(X,Z)$ by \cite{CP} and \cite{A}.

\begin{rmk}\label{rmk:imperfect}
In Lemma \ref{imperfect} we will need to work on a more general setting than the one above. 
In this setting, $X$ is allowed to be a normal, excellent scheme over a possibly imperfect field. 
All the above definitions can be extended to this setting. 
For details, in particular for the precise definitions of $K_X$, we refer to \cite[Definition 2.4, Definition 2.8]{Kollar}.
\end{rmk}

\subsection{MMP over a perfect field}

For later use, we summarize some results on the minimal model theory. 
First we collect some results which work in an arbitrary positive characteristic. 

\begin{lem}[Relative pseudo-effective cone theorem]\label{lemma-cone}
Let $k$ be a perfect field of characteristic $p>0$. 
Let $X$ be a $\Q$-factorial threefold 
which is projective over a separated scheme $U$ of finite type over $k$. 
Let $\Delta$ be a $\Q$-divisor on $X$ whose coefficients are contained in $[0, 1]$. 
Assume that $K_X+\Delta$ is pseudo-effective over $U$. 
Then, for every ample $\Q$-divisor $A$, 
there exists a finitely many curves $C_1, \cdots, C_r$ such that 
$$\overline{NE}(X/U)=\overline{NE}(X/U)_{K_X+\Delta+A \geq 0}+\sum_{i=1}^r \mathbb R_{\geq 0}[C_i].$$
\end{lem}

\begin{proof}
We can perturb $\Delta$, by using an ample $\Q$-divisor $A$, and 
we may assume that 
$$K_X+\Delta\sim_{\Q, U}E$$
where $E$ is an effective $\Q$-divisor. 
Then, the assertion is proved by the same proof as in \cite[Proposition 0.6]{Keel} 
(cf.\ \cite[Theorem~6.6]{Tanaka:behavior}). 
\end{proof}

\begin{lem}[Relative Keel's theorem]\label{relative-Keel}
Let $k$ be a perfect field of characteristic $p>0$. 
Let $\pi:X \to U$ be a projective morphism 
from a normal variety $X$ to a separated scheme $U$ of finite type over $k$. 
Let $L$ be a $\pi$-nef $\Q$-Cartier $\Q$-divisor on $X$. 
Assume that $L=A+E$ where $A$ is a $\pi$-ample $\Q$-Cartier $\Q$-divisor 
and $E$ is an effective $\Q$-divisor on $X$. 
If $L|_{\Supp E}$ is $\pi$-semi-ample, then $L$ is $\pi$-semi-ample. 
\end{lem}

\begin{proof}
It is proved by the same proof as in \cite[Lemma~3.3]{CMM} (cf.\ \cite[Proposition~2.7]{Xu}).
\end{proof}

\begin{lem}[Pl-contraction theorem]\label{lemma-cont}
Let $k$ be a perfect field of characteristic $p>0$. 
Let $(X, \Delta)$ be a $\Q$-factorial dlt threefold 
with a projective morphism $\pi:X\to U$ to a separated scheme $U$ of finite type over $k$. 
Let $L$ be a $\pi$-nef and $\pi$-big Cartier divisor on $X$ 
such that $\overline{NE}(X/U) \cap L^{\perp}=:R$ is a $(K_X+\Delta)$-negative extremal ray 
of $\overline{NE}(X/U)$. 
Assume the following two conditions. 
\begin{itemize}
\item{$S\cdot R<0$ for some irreducible component $S$ of $\llcorner \Delta\lrcorner$. }
\item{The normalization of $S$ is a universal homeomorphism.}
\end{itemize}
Then, $L$ is $\pi$-semi-ample. 
\end{lem}

\begin{proof}
We may assume that $(X, \Delta)$ is plt with $S=\llcorner \Delta\lrcorner$. 
By the assumption, $A:= L-\epsilon S$ is $\pi$-ample for small $\epsilon>0$. 
Since $L= A+\epsilon S$, 
it suffices to show that $L|_S$ is $\pi$-semi-ample by Lemma~\ref{relative-Keel}. 
Since the normalization $S^N \to S$ is a universal homeomorphism, 
it is enough to prove that $L|_{S^N}$ is $\pi$-semi-ample (cf.\ \cite[Lemma 1.4]{Keel}). 
By adjunction, the new pair $(S^N, \Delta_{S^N})$, defined by 
\[
(K_X+\Delta)|_{S^N}=K_{S^N}+\Delta_{S^N}, 
\]
is klt, where $\Delta_{S^N}=\mathrm{Diff}_{S^N}(\Delta)$ is the different 
in the sense of \cite[Section 4.1]{Kollar}. 
Since 
$$mL|_{S^N}-(K_{S^N}+\Delta_{S^N})=(mL-(K_X+\Delta))|_{S^N}$$
is $\pi$-ample for $m \gg 0$ by the assumption, 
it follows that $L|_{S^N}$ is $\pi$-semi-ample by \cite[Theorem 6.9]{Tanaka}. 
\end{proof}

\begin{lem}[R1 for plt centers]\label{R1}
Let $k$ be a field. 
Let $(X, S+B)$ be a plt pair over $k$ such that $\llcorner S+B\lrcorner=S$. 
Then, $S$ is regular in codimension one, that is, 
there exists an open subset $U \subset S$ such that 
${\rm codim}_S(S\setminus U)\geq 2$ and that $U$ is regular. 
\end{lem}

\begin{proof} 
Set $\nu:S^N \to S$ to be the normalization. 
Then the log pair $(S_N, \Delta_{S^N})$, defined by $(K_X+S+B)|_{S^N}=K_{S^N}+\Delta_{S^N}$, is klt, where  $\Delta_{S^N}=\mathrm{Diff}_{S^N}(\Delta)$ is the different in the sense of \cite[Section 4.1]{Kollar}.
Then $\llcorner \Delta_{S^N}\lrcorner =0$. 
If $S$ is not regular in codimension one, 
then the corresponding divisor to 
the codimension one part of the conductor ideal appears in $\llcorner \Delta_{S^N}\lrcorner$ 
by \cite[Proposition~4.5~(1)]{Kollar}. 
This contradicts $\llcorner \Delta_{S^N}\lrcorner =0$. 
\end{proof}

\begin{thm}[Special termination]\label{special termination}
Let $k$ be a perfect field and $Z$ a quasi-projective variety over $k$. 
Let $(X, S+B)$ be a $\mathbb Q$-factorial
dlt threefold, projective over $Z$, such that $\llcorner S+B\lrcorner=S$. 
Then every sequence of  $(K_X+S+B)$-flips over $Z$   terminates around $S$. 
\end{thm}
\begin{proof}
The arguments of the proof in \cite[Theorem 4.2.1]{Fujino:ST} 
work by replacing $S_i$ with the normalization $S^{\nu}_i$ of $S_i$.
\end{proof}

\medskip

Second, we summarize results in characteristic $p>5$. 
Almost all the results are known for the case where $k$ is algebraically closed. 
We shall run an MMP over a finite field in Section~\ref{section-rational-pt}, 
hence we need to generalize them over perfect fields. 

\begin{rmk}\label{rem:perfect_to_closure}
We summarize here what kind of properties are preserved under the base change to the algebraic closure. 
\begin{enumerate}
\item Let $X$ be a normal variety over a perfect field. 
Then its base change $X' := X \times _k \overline{k}$ to the algebraic closure $\overline{k}$ is also normal. 
In particular, each connected component of $X'$ is reduced and irreducible. 
Further, if a pair $(X, \Delta)$ is log canonical (resp.\ klt, plt, dlt), then so is 
$(X', \Delta' := \Delta \times _k \overline{k})$ 
(see Proposition~2.15 in \cite{Kollar} and Warning below it). 
However, the $\mbQ$-factoriality is not preserved in general. 
Let $L$ be a line bundle on $X$ and let $L'$ be its pull-back to $X'$. 
Then, $L$ is ample (resp.\ nef, big, semi-ample) if and only if so is $L'$. 

\item 
Let $k \subseteq k'$ be a field extension and $\pi:X \to S$ a morphism of schemes over $k$, 
and let $L$ be a line bundle on $X$. 
Denote the base change 
$X\times_{k} k' \to S \times_{k} k'$ by $\pi':X' \to S'$ and denote the pull-back
$L \otimes_{\mathcal{O}_X} \mathcal{O}_{X'}$ by $L'$. 
Then the following hold:
\begin{enumerate}
\item[(a)]$L$ is base point free over $S$ if and only if so is $L'$ over $S'$, and
\item[(b)] $\bigoplus_{l \geq 0} \pi_*(L^{\otimes l})$ is a finitely generated algebra over $\mathcal{O}_{S}$ 
if and only if so is $\bigoplus_{l \geq 0} \pi'_*(L'^{\otimes l})$ over $\mathcal{O}_{S'}$ 
(\cite[Proposition 2.7.1]{EGAIV2}). 
\end{enumerate}
Indeed, we have $\pi_*(L^{\otimes l}) \otimes_{\mathcal{O}_X} \mathcal{O}_{X'} \cong  \pi'_*(L'^{\otimes l})$ 
by the flat base change under $\Spec k' \to \Spec k$. 
\end{enumerate}
\end{rmk}

\begin{lem}\label{perturb/finite-field}
Let $(X, \Delta)$ be a klt pair over a perfect field with $\dim X \leq 3$. 
Let $A$ be a nef and big $\Q$-Cartier $\Q$-divisor. 
Then there exists an effective $\Q$-divisor $A' \sim_{\Q} A$ such that $(X, \Delta+A')$ is klt. 
\end{lem}

\begin{proof}
By the existence of a log resolution of $(X, \Delta)$, 
we may assume that $X$ is smooth and quasi-projective, 
and that $\Delta$ is simple normal crossing, and $A$ is ample. 
If $k$ is infinite, then we can apply the usual Bertini theorem. 
If $k$ is a finite field, we can apply the same argument as in \cite[Theorem~1.1 or 1.2]{Poonen}. 
\end{proof}

\begin{thm}[Relative base point free theorem]\label{thm:bpf}
Let $k$ be a perfect field of characteristic $p>5$. 
Let $(X, \Delta)$ be a klt threefold over $k$, equipped with a projective morphism $f : X \to Z$ to a quasi-projective variety $Z$ over $k$ such that 
$f_* \mathcal{O}_X = \mathcal{O}_Z$. 
If $L$ is a $f$-nef Cartier divisor  such that 
$L-(K_X+\Delta)$ is $f$-nef and $f$-big, 
then $L$ is $f$-semi-ample. 
\end{thm}

\begin{proof}
By Remark \ref{rem:perfect_to_closure}, we may assume that $k$ is algebraically closed. 
When $Z$ is projective, the assertion follows from \cite[Theorem~1.2]{BW}. 

For the general case, we may assume that $X$ is $\mathbb{Q}$-factorial by \cite[Theorem 1.6]{Birkar}. 
Let $A$ be an $f$-ample effective $\mbQ$-divisor such that 
\[
A \sim _{f, \mbQ} L - (K_X + \Delta)
\]
and $(X, \Delta+A)$ is klt.  
We take a compactification $\overline{f}:\overline{X} \to \overline{Z}$ of $f:X \to Z$ 
such that $\overline{X}$ and $\overline{Z}$ are projective normal threefolds. 
In particular, $\overline{f}_* \mathcal{O}_{\overline{X}} = \mathcal{O}_{\overline{Z}}$. 
Let $\varphi: \widetilde{X} \to \overline{X}$ be a log resolution 
of $(\overline X \setminus X) \cup \Supp (\Delta+A)$ and we denote the composition by  
$$\widetilde{f}:\widetilde{X} \xrightarrow{\varphi} \overline X \xrightarrow{\overline f} \overline{Z}.$$ 
Let $\widetilde{\Delta}$ and $\widetilde{A}$ be the strict transforms on $\widetilde{X}$ of the closures of $\Delta$ and $A$, respectively. 
We set $E:=(1-\epsilon)\sum E_i$, where $\sum E_i$ is the sum of all the $\varphi$-exceptional prime divisors 
and $\epsilon$ is a sufficiently small rational number. 
Run a $(K_{\widetilde{X}}+\widetilde{\Delta}+\widetilde{A}+E)$-MMP over $\overline{X}$ (\cite[Theorem~1.6]{BW}). 
Then the end result $\varphi':(\widetilde{X}', \widetilde{\Delta}'+\widetilde{A}'+ E') \to \overline{X}$ is isomorphic 
over $X$ (hence $E'$ is supported on $\widetilde{X}' \setminus (\varphi ') ^{-1}(X)$) 
since $X$ is $\mbQ$-factorial klt and $\epsilon$ is sufficiently small. 
Next run a $(K_{\widetilde{X}'}+\widetilde{\Delta}'+\widetilde{A}'+ E')$-MMP over $\overline{Z}$ (\cite[Theorem~1.6]{BW}). 
Since $K_{\widetilde{X}'} +\widetilde{\Delta}'+\widetilde{A}'+ E'$ is already nef over $Z$ by the assumption, 
the resulting model $f'':(\widetilde{X}'', \widetilde{\Delta}''+\widetilde{A}''+ E'') \to \overline{Z}$ is isomorphic to 
$f':(\widetilde{X}', \widetilde{\Delta}'+\widetilde{A}'+ E') \to \overline{Z}$ over $Z$. 
In particular, $f'':(\widetilde{X}'', \widetilde{\Delta}''+\widetilde{A}''+ E'') \to \overline{Z}$ is a compactification of 
$f:(X, \Delta+A) \to Z$. 
Further, it follows that the $\mbQ$-divisor $K_{\widetilde{X}''}+\widetilde{\Delta}''+\widetilde{A}''+ E''$ is $f''$-nef. 

Thus, in order to show the $f$-semi-ampleness of $K_X + \Delta + A$, 
it is enough to show the $f''$-semi-ampleness of 
$K_{\widetilde{X}''}+\widetilde{\Delta}''+\widetilde{A}''+ E''$
(since $E''$ is supported on $\widetilde{X}'' \setminus (f'')^{-1}(Z)$). 
Since $\widetilde{A}''$ is $f''$-big, 
there exist an $f''$-ample $\Q$-divisor $M$ and an effective $\Q$-divisor $\Gamma$ such that 
$\widetilde{A}'' \sim_{f'',\mathbb{Q}} M+\Gamma$. Take a rational number $\delta>0$ such that 
$(\widetilde{X}'', \widetilde{\Delta}''+(1-\delta)\widetilde{A}''+ E''+\delta \Gamma) $ is klt. 
Then the $f''$-semi-ampleness of 
$$
K_{\widetilde{X}''}+\widetilde{\Delta}''+\widetilde{A}''+ E''
\sim_{f'', \mathbb{Q}}
K_{\widetilde{X}''}+ \widetilde{\Delta}''+(1-\delta)\widetilde{A}''+ E''+\delta \Gamma+\delta M$$
follows from the projective case (\cite[Theorem~1.2]{BW}).  
\end{proof}

\begin{thm}[Existence of flips]\label{thm:flip}
Let $k$ be a perfect field of characteristic $p>5$. 
Let $f:(X, \Delta) \to Z$ be a flipping contraction from a $\mbQ$-factorial dlt threefold $(X, \Delta)$. 
Then a flip of $f$ exists. 
\end{thm}

\begin{proof}
We may assume that $(X, \Delta)$ is klt. 
The assertion follows from \cite[Theorem 1.3]{Birkar} and Remark~\ref{rem:perfect_to_closure}~(2). 
\end{proof}

In dimension three and in characteristic $p > 5$, the normality of plt centers is known 
(cf.\ Lemma \ref{R1}). 

\begin{thm}[Normality of plt centers]\label{thm:normality}
Let $k$ be a perfect field of characteristic $p>5$. 
Let $(X, \Delta)$ be a $\mbQ$-factorial dlt threefold over $k$. 
Then every irreducible component of $\llcorner \Delta\lrcorner$ is normal. 
\end{thm}

\begin{proof}
We may assume that $\Delta$ is a prime divisor, in particular the pair is plt. 
Then, the base change $(X \times_k \overline k, \Delta \times_k \overline k)$ to the algebraic closure is also plt. 
Since $\Delta \times_k \overline k$ is normal by Theorem~3.1 and Proposition~4.1 in \cite{HX}, so is $\Delta$. 
\end{proof}

\begin{thm}[MMP with scaling]\label{mmp/a field}
Let $k$ be a perfect field of characteristic $p>5$, and let
$(X, \Delta)$ be a projective $\Q$-factorial klt threefold over $k$, 
equipped with a morphism $f: X \to Z$ to a projective variety $Z$ over $k$ with $f_*\MO_X=\MO_Z$. 
Let $C$ be an effective $\Q$-divisor on $X$ which satisfies the following properties:
\begin{enumerate}
\item{$(X, \Delta+C)$ is klt,}
\item{$C$ is $f$-big, and}
\item{$K_X+\Delta+C$ is $f$-nef.}
\end{enumerate}
Then, we can run a $(K_X+\Delta)$-MMP over $Z$ with scaling of $C$ and it terminates. 
\end{thm}

\begin{proof}
By Theorem \ref{thm:bpf} and Theorem \ref{thm:flip}, 
the only problem to show is the	 termination of flips. 

First, we show the assertion under the assumption that 
$k$ is algebraically closed and that $Z=\Spec\,k$. 
If $K_X+\Delta$ is not pseudo-effective, then the assertion follows from \cite[Theorem~1.5]{BW} 
because in this case, a $(K_X+\Delta)$-MMP with scaling of $C$ can be considered 
as a $(K_X+\Delta+\epsilon C)$-MMP with scaling of $C$ for some $\epsilon>0$. 
If $K_X+\Delta$ is pseudo-effective, 
then we can apply the same argument as in \cite[Proposition~4.4]{BW}. 

Second, we work on the case where $k$ is algebraically closed field 
but we do not impose any additional assumption on $Z$. 
Set $C' = C+f^*H$ for $H$ a sufficiently ample divisor on $Z$. 
Then $C' = C+f^*H$ is globally big and $K_X+\Delta+C'$ is globally nef by the cone theorem. 
Moreover, by \cite[Theorem\,1]{Tanaka:semiample}, $(X, \Delta')$ becomes klt again for some effective $\Q$-divisor $\Delta'$ 
such that $\Delta' \sim_{\Q}\Delta+f^*H$. 
Since $H$ is sufficiently ample, 
any $(K_X+\Delta)$-MMP over $Z$ with scaling of $C$ can be considered as 
a $(K_X+\Delta')$-MMP over $\Spec\,k$ with scaling of $C'$. 
Then the termination follows from the case where $Z=\Spec\,k$. 

Third, let us go back to the general case, where $k$ is perfect. 
Let 
\[
(X, \Delta)=(X_1, \Delta_1) \dasharrow (X_2, \Delta_2) \dasharrow \cdots
\]
be a $(K_X+\Delta)$-MMP over $Z$ with scaling of $C$. 
Suppose that this sequence consists of flips.  
Take the base change to the algebraic closure. 
The new log pairs after the base change might be non-$\mbQ$-factorial. 
However, by \cite[Remark 2.9]{Birkar:sLMMP}, taking a small $\mbQ$-factorialization for each log pair, 
we can construct a sequence of an MMP with scaling of some $C'$ which satisfies the same properties 
as in (1)--(3). 
More precisely, There exists a small $\mbQ$-factorialization $(X_i ', \Delta _i ')$ such that 
the new sequence
\[
(X', \Delta')=(X_1 ', \Delta_1 ') \dasharrow (X_2', \Delta_2') \dasharrow \cdots
\]
becomes an MMP with scaling $C'$ that is the pullback of $C$ to $X '$. 
Hence the termination of flips follows from the case of algebraically closed fields. 
The argument in \cite[Remark 2.9]{Birkar:sLMMP} is working on the characteristic zero, 
but it also works in our setting since the small $\mbQ$-factorialization exists by \cite[Theorem 1.6]{Birkar} and 
the MMP necessary in the argument is known to exist by \cite[Theorem\,1.6]{BW}. 
\end{proof}

We can drop the projectivity assumption on $Z$. 

\begin{thm}[Relative MMP]\label{thm:relmmp}
Let $k$ be a perfect field of characteristic $p>5$, and let
$(X, \Delta)$ be a $\Q$-factorial klt threefold over $k$, 
equipped with a projective morphism $f: X \to Z$ to a quasi-projective variety $Z$ over $k$. 
Then, we can run some $(K_X+\Delta)$-MMP over $Z$ and it terminates. 
More precisely, there exists a sequence 
\[
(X, \Delta)=:(X_1, \Delta_1) \dasharrow (X_2, \Delta_2) \dasharrow \cdots \dasharrow (X_N, \Delta_N)
\]
consisting of divisorial contractions over $Z$ and flips over $Z$ such that 
$K_{X_N}+\Delta_N$ is nef over $Z$ or $(X_N, \Delta_N)$ has a Mori fiber space structure over $Z$. 
\end{thm}

\begin{proof}
We may assume $f_*\MO_X=\MO_Z$. 
We can construct a projective $\Q$-factorial klt compactification $\overline f:(\overline X, \overline \Delta) \to \overline Z$ 
of $f:(X, \Delta) \to Z$ in the following way. 
Let $X' \to Z'$ be an arbitrary projective compactification of $X \to Z$, 
where $X'$ and $Z'$ are projective normal varieties. 
Set $\Delta'$ to be the closure of $\Delta$ 
(note here that $K_{X'}+\Delta'$ may not be $\Q$-Cartier). 
Let $\varphi': Y' \to X'$ be a resolution of singularities 
such that the support of $\varphi'^{-1}(\Supp(\Delta'))  \cup \Ex(\varphi')$ 
is a simple normal crossing divisor. 
Set $Y:=\varphi'^{-1}(X)$ and write
\[
\varphi:= \varphi'|_{Y}:Y \to X.
\]
We can write 
$$K_Y+\varphi_*^{-1}\Delta +(1-\epsilon)E=\varphi^*(K_X+\Delta)+E',$$
where $E$ is the sum of the $\varphi$-exceptional prime divisors, 
$\epsilon$ is a rational number with $0<\epsilon<1$, and 
$E'$ is a $\varphi$-exceptional effective $\Q$-divisor satisfying 
$\Supp E'= \Ex(\varphi)$. 
For the closure $\Gamma$ of $\varphi_*^{-1} \Delta +(1-\epsilon)E$ in $Y'$, 
it follows from the choice of $\varphi'$ that $(Y', \Gamma)$ is klt. 
We run a $(K_{Y'}+\Gamma)$-MMP over $X'$ (Theorem \ref{mmp/a field}), 
and set $(\overline{X}, \overline{\Delta})$ to be the end result. 
Then $(\overline{X}, \overline{\Delta})$ is a $\Q$-factorial klt pair. 
Further, by the negativity lemma and the $\Q$-factoriality of $X$, 
the pair $(\overline{X}, \overline{\Delta})$ is a compactification of $(X, \Delta)$.

We run a $(K_{\overline X}+\overline \Delta)$-MMP over $\overline{Z}$ with scaling of some ample divisor 
and it terminates by Theorem~\ref{mmp/a field}. 
Then the assertion follows by restricting this MMP to $Z$. 
\end{proof}

\begin{thm}[$\mbQ$-factorialization]
Let $k$ be a perfect field of characteristic $p>5$, and let
$(X, \Delta)$ be a klt threefold over $k$. 
Then there exists a small projective birational morphism $f:Y \to X$ such that $Y$ is $\Q$-factorial. 
\end{thm}

\begin{proof}
Let $g:Z \to X$ be a log resolution and let $E$ be the sum of the $g$-exceptional prime divisors. 
For a small $\epsilon>0$, we can run a $(K_Z+g_*^{-1}\Delta+ (1- \epsilon)E)$-MMP over $X$, and 
it terminates with a $\mbQ$-factorial model $Y$ (Theorem~\ref{thm:relmmp}). 

Since $E$ is contracted by this MMP by the negativity lemma, 
$Y \to X$ turns out to be small. 
\end{proof}

\subsection{Special fibers of log Fano contractions}

The following proposition shall be used repeatedly in this paper. 
This proposition allows us to assume that a closed fiber of a log Fano contraction 
is a log Fano variety in many situations.

\begin{prop}\label{MFS-special}
Let $k$ be a perfect field of characteristic $p>5$. 
Let $f:X \to Z$ be a projective morphism of normal quasi-projective varieties over $k$ 
with the following properties. 
\begin{itemize}
\item{$(X, \Delta)$ is a klt threefold.}
\item{$0<\dim Z \leq 3$.}
\item{$f_*\MO_X=\MO_Z$ and $-(K_X+\Delta)$ is $f$-nef and $f$-big.}
\end{itemize}
Fix a closed point $z \in Z$. 
Then there exists a commutative diagram of quasi-projective normal varieties 
$$\begin{CD}
W @>\psi >> Y\\
@VV\varphi V @VVgV\\
X @>f>> Z
\end{CD}$$
and an effective $\Q$-divisor $\Delta_Y$ on $Y$ which satisfy the following properties. 
\begin{enumerate}
\item{$(Y, \Delta_Y)$ is a $\Q$-factorial plt threefold and $-(K_Y+\Delta_Y)$ is $g$-ample.}
\item{$-\llcorner \Delta_Y \lrcorner$ is $g$-nef and $g^{-1}(z)_{\red} = \llcorner \Delta_Y\lrcorner$.}
\item{$W$ is a smooth threefold, and both of $\varphi$ and $\psi$ are projective birational morphisms.}
\item{If $X=Z$, then $g(\Ex(g))$ is a finite set. 
More precisely, $g$ is isomorphic 
outside $z$ and the non-$\mbQ$-factorial points on $Z$. }
\end{enumerate}
In particular, $\llcorner \Delta_Y \lrcorner$ is a surface of Fano type 
by Theorem \ref{thm:normality} and adjunction
(See Definition \ref{defi:fanotype} for the definition of Fano type). 
\end{prop}

\begin{proof}
We may assume that $-(K_X+\Delta)$ is $f$-ample. 
Let $h:V \to X$ be a log resolution of $(X, \Delta)$. 
Possibly replacing $V$ with a higher model, 
we can find a Cartier divisor $D_Z$ on $Z$ such that 
\[
0\sim_Z h^*f^*D_Z=M+F
\]
where $M$ is a base point free Cartier divisor and $F=\sum f_iE_i$ is an effective divisor 
with $\Supp F=\Supp V_z$, where $V_z$ is the fiber of $V \to Z$ over $z$ (note that $F$ may not be $h$-exceptional). 
We may assume that $\Supp(h^{-1}(\Delta) \cup F \cup \Ex(h))$ is simple normal crossing. 
We can write 
\[
K_V+h_*^{-1}\Delta=h^*(K_X+\Delta)+\sum a_iE_i
\]
with $a_i>-1$. 
Since $-h^*(K_X+\Delta)$ is nef and big over $Z$, 
we can write $-h^*(K_X+\Delta) =A+G$, where $A$ is an ample $\Q$-divisor over $Z$ 
and $G$ is an effective $h$-exceptional $\Q$-divisor. 
Let $G=\sum_{i}g_iE_i$ be the decomposition by prime divisors of $G$. 
Then we can take a small positive rational number $\epsilon$ such that $a_i-\epsilon g_i >-1$ for any $i$.
Then, we have 
\[
K_V+h_*^{-1}\Delta+\epsilon A-(1-\epsilon)h^*(K_X+\Delta)=\sum (a_i-\epsilon g_i) E_i.
\]
Since $M+F \sim_Z 0$, we can find $\lambda \in \Q_{>0}$ with  
\[
K_V+h_*^{-1}\Delta+ \big( \epsilon A- (1-\epsilon) h^*(K_X+\Delta)+\lambda M \big) \sim_{\Q, Z} \sum (a_i-\epsilon g_i-\lambda f_i)E_i,
\]
such that $a_i-\epsilon g_i-\lambda f_i \geq -1$ holds for any $i$ and at least one index attains $-1$. 

Let $S_V$ be a prime divisor which attains $-1$ in the above. 
Since $\Supp F=\Supp V_z$, it follows that $S_V \subset V_z$. 
Let $\{ E_i \} _{i \in I'}$ be the set of the other prime divisors contained in $V_z \cup \Ex (h)$. 
Since $\epsilon A- (1- \epsilon) h^*(K_X+\Delta)+\lambda M$ is ample over $Z$, 
we may find an ample $\mbQ$-divisor $A'$ such that 
\[
K_V+h_*^{-1}\Delta+A'+S_V+\sum_{i \in I'} E_i
\sim_{\Q, Z} \sum_{i \in I'} b_iE_i
\]
with $b_i > 0$. 
Note that $S_V$ and $\sum_{i \in I'} b_iE_i$ may not be $h$-exceptional, but 
that $\Supp \sum_{i \in I'} b_iE_i$ contains all the $h$-exceptional prime divisors 
whose images in $X$ are not contained in $X_z$. 
We can assume that $A'$ is effective and $\llcorner A'\lrcorner=0$, and that
$h_*^{-1}\Delta+A'+ S_V + \sum_{i \in I'} E_i$
has a simple normal crossing support. 
Set $B_V:=h_*^{-1}\Delta+A'$, which is big. 

We shall run an MMP three times to get a desired $g: Y \to Z$. 

\vspace{2mm}

\noindent
\underline{STEP1.}\ \ 
First, we run a $(K_V+B_V+S_V+\sum_{i \in I'} E_i)$-MMP over $X$ 
(cf. Lemma~\ref{lemma-cone}, Theorem~\ref{thm:bpf}, and Theorem~\ref{thm:flip}). 
Since this MMP is also $(\sum_{i \in I'} b_iE_i)$-MMP, 
the curves contracted by this MMP is contained in 
$$\bigcup _{i \in I'} \Supp E_i.$$
Hence this MMP terminates by the special termination (Theorem \ref{special termination}), and
ends with a minimal model 
$(V', B_{V'} + S_{V'}+\sum_{i \in I'} E_i')$ over $X$, that is 
\[
K_{V'}+B_{V'}+S_{V'}+\sum_{i \in I'} E_i' \sim_{\Q, Z} \sum_{i \in I'} b_iE_i'.
\]
is nef over $X$. Note that $E_i'=0$ if $E_i$ is contracted by this MMP. 
Since this MMP contracts a curve contained in $\bigcup _{i \in I'} \Supp E_i$, 
this MMP does not contract $S_V$. 
We can also see that $\Supp(\sum b_iE_i') \subset V'_z$ 
by the negativity lemma for $V' \setminus V'_z$ over $X \setminus X_{z}$. 

\vspace{2mm}

\noindent
\underline{STEP2.}\ \ 
Second, we run a $(K_{V'}+B_{V'}+S_{V'}+\sum _{i \in I'} E_i')$-MMP over $Z$. 
By the same reason as in STEP1, this MMP terminates and ends with 
a minimal model $(V'', B_{V''} + S_{V''}+\sum _{i \in I'} E_i'')$ over $Z$, that is, 
\[
K_{V''}+B_{V''}+S_{V''}+\sum _{i \in I'} E_i'' \sim_{\Q, Z} \sum _{i \in I'} b_iE_i''
\]
is nef over $Z$. 
Further, this MMP does not contract $S_{V'}$. 

We shall show that $\sum _{i \in I'} b_iE_i'' \sim_{\Q, Z} 0$. 
Let $G_Z$ be an effective Cartier divisor on $Z$ passing through $z$, 
and let $G_{V''}$ be its pull-back to $V''$. 
Since $\Supp(\sum b_iE_i') \subset V'_z$, there exists $\mu \ge 0$ such that 
\[
-\mu G_{V''} + \sum _{i \in I'} b_iE_i'' \le 0
\]
and at least one coefficient of $E_i''$ in $-\mu G_{V''} + \sum _{i \in I'} b_iE_i''$ is zero. 
If 
\[
-\mu G_{V''} + \sum _{i \in I'} b_iE_i'' \sim _{\Q, Z} \sum _{i \in I'} b_iE_i'' \not\sim_{\Q, Z} 0, 
\]
there exists a prime divisor $E_{i_0}''$ whose coefficient in 
$-\mu G_{V''} + \sum _{i \in I'} b_iE_i''$ is zero 
but one of the prime divisors 
that meet $E_{i_0}''$ but are different from $E_{i_0}''$ has negative coefficient. 
Then a general curve on $E_{i_0}''$ has negative intersection number with $-\mu G_{V''} + \sum _{i \in I'} b_iE_i''$, 
and it contradicts the nefness of $-\mu G_{V''} + \sum _{i \in I'} b_iE_i''$ over $Z$. 
Hence it follows that $\sum _{i \in I'} b_iE_i'' \sim_{\Q, Z} 0$ by contradiction. 

\vspace{2mm}

\noindent
\underline{STEP3.}\ \ 
Third, we run a 
$(K_{V''}+B_{V''}+ \frac{1}{2}S_{V''}+\sum _{i \in I'} E_i'')$-MMP over $Z$. 
Since $\sum b_iE_i'' \sim_{\Q, Z} 0$, 
this MMP can be considered as an MMP for a klt pair, 
hence we may run this MMP and it terminates. 
By 
\[
K_{V''}+B_{V''}+\frac{1}{2}S_{V''}+\sum _{i \in I'} E_i'' \sim_{\Q, Z} -\frac{1}{2}S_{V''},
\]
this MMP is $S_{V''}$-positive, hence does not contract $S_{V''}$. 
Since $S_{V''}$ is contained in the fiber $V_z ''$ over $z$, 
we can write $-\frac{1}{2}S_{V''} \sim_{\Q, Z} L$ for some effective $\Q$-divisor $L$ on $V''$. 
Therefore, the end result 
$(Y, B_{Y} + \frac{1}{2}S+\sum _{i \in I'} E_{i, Y})$ is a minimal model equipped with $g:Y \to Z$. 
Since $-S$ is $g$-nef and $S \subset g^{-1}(z)_{\red}$, it follows that $S=g^{-1}(z)_{\red}$. 
In particular, all the prime divisors in $\sum _{i \in I'} E_i''$ are contracted in this MMP. 
Therefore, we obtain 
$$K_Y+B_Y+S \sim_{\Q, Z} 0.$$
Since $B_V$ is big, so is $B_Y$. 
We can find a $\Q$-divisor $B \geq 0$ such that 
$(Y, S+B)$ is plt and $-(K_Y+S+B)$ is ample. 
It completes the proof of (1)--(3). 

\vspace{2mm}

Finally, we prove (4). Suppose $X = Z$. 
Then STEP2 is unnecessary. 
In STEP1, all the exceptional divisors of $V' \to X$ are contained in the fiber $V'_z$ 
(recall that we saw $\Supp(\sum b_iE_i') \subset V'_z$). 
Then  $V' \to X$ turns out to be isomorphic outside $\{z, x_1, \cdots, x_r\}$, 
where $x_1, \cdots, x_r $ are the non-$\Q$-factorial points of $X=Z$. 
Therefore, the result $Y \to X=Z$ is also isomorphic outside $\{z, x_1, \cdots, x_r\}$. 
\end{proof}

\begin{rmk}
For the case $X=Z$ in Proposition~\ref{MFS-special}, 
$Y$ is very similarly constructed as the Koll{\'a}r component in \cite[Lemma 1]{Xu2}. 
\end{rmk}

Proposition \ref{MFS-special} holds except for (4) 
in any dimension if we assume the existence of log resolution and an MMP with scaling.

\subsection{Witt vectors}\label{subsection:witt}
We briefly recall the definition of the ring of Witt vectors. 
For more details we refer the reader \cite[II, Section 6]{Serre}. 

Let $k$ be a perfect field of characteristic $p > 0$. 
We define a functor $W$ from the category of $\mathbb F_p$-algebras to 
the category of $\mbZ$-algebras: 
\[
W: {\bf Alg}_{/ \mathbb F_p} \to {\bf Alg}_{/ \mbZ}; \quad
A \mapsto WA
\]
Although we can define $WA$ for any ring $A$, 
we restrict ourselves to treat $\mathbb F_p$-algebras. 
Firstly, as a set, we define $WA$ to be the product $A^{\mbN}$ of countably many copies of $A$: 
\[
WA=A\times A\times A\times\cdots=\{(a_0, a_1, a_2, \ldots)\,|\, a_i\in A\}. 
\]
The zero and the unity element are 
\[
0_{WA}=(0_A, 0_A, 0_A, \ldots), \quad 1_{WA}=(1_A, 0_A, 0_A, \ldots). 
\]
We define the ring structure as follows: 
For two vectors $a = (a_0, a_1, \ldots), b = (b_0, b_1, \ldots)$, 
we set 
\begin{align*}
a + b &:= (S_0(a_0,b_0), S_1(a_0, a_1,b_0, b_1), S_2(a_0, a_1, a_2 ,b_0, b_1, b_2) \ldots), \\
a \cdot b &:= (P_0(a_0,b_0), P_1(a_0, a_1,b_0, b_1), P_2(a_0, a_1, a_2 ,b_0, b_1, b_2) \ldots), 
\end{align*}
where $S_i, P_i \in \mbZ [X_0, \ldots , X_i, Y_0, \ldots , Y_i]$ are 
the unique polynomials satisfying the following law: 
\begin{itemize}
\item When we set polynomials $w_n$ as $w_n(X_0, \cdots, X_n) := \sum _{0 \le i \le n} p^i X_i ^{p-i}$ for each $n \ge 0$, 
$S_i$ and $P_i$ satisfy the following equations 
\begin{align*}
w_n (S_0, \ldots , S_n) &= w_n(X_0, \ldots , X_n) + w_n(Y_0, \ldots , Y_n), \\
w_n (P_0, \ldots , P_n) &= w_n(X_0, \ldots , X_n)  w_n(Y_0, \ldots , Y_n)
\end{align*} 
for each $n \ge 0$.  
\end{itemize}
We have, for example, $W(\mathbb F_p)=\mathbb Z_p$. 
In general, $W(k)$ is a complete discrete valuation ring of mixed characteristic with $W(k)/pW(k)=k$. 
In particular, $W(k) _{\mathbb{Q}}:=W(k) \otimes_{\mathbb{Z}} \mathbb{Q}$ is a field. 

For a ring homomorphism $\varphi:A \to B$, 
we define a ring homomorphism $W\varphi$ by
\[
W \varphi:WA \to WB; \quad 
(a_0, a_1, a_2, \ldots) \mapsto (\varphi(a_0), \varphi(a_1), \varphi(a_2), \ldots).
\]
We note that if $\varphi$ is injective (resp.\ surjective), so is $W\varphi$. 
In particular, if $A$ is a nonzero $k$-algebra, then the natural ring homomorphism $k \to A$ is injective 
and so we have an injective ring homomorphism $W(k) \to WA$. 
Hence $WA$ has a $W(k)$-algebra structure.

For $n \ge 1$, we can also define a ring $W_n A$ which is equal to $A^{n}$ as a set and 
with the same operation as $WA$. 
We have a natural projection 
\[
R: W_{n+1} A \to W_n A; \quad (a_0, \ldots, a_{n-1}, a_{n}) \mapsto (a_0, \ldots, a_{n-1}). 
\]
Then the ring $W A$ is isomorphic to the projective limit $WA = \varprojlim_n W_n A$. 
In general, $W_n (k)$ is an Artinian ring.

We can define maps $V$ and $F$ as follows:
\begin{align*}
V: W_n A \to W_{n+1} A;& \quad (a_0, a_1, \ldots, a_{n-1}) \mapsto (0, a_0, a_1, \ldots, a_{n-1}), \\
F: W_n A \to W_{n} A;& \quad (a_0, a_1, \ldots, a_{n-1}) \mapsto (a_0 ^p, a_1^p, \ldots, a_{n-1}^p). 
\end{align*}
It is well-known that $F$ is a homomorphism of rings and that $V$ is a homomorphism of additive groups. 
We can also define $V$ and $F$ as endomorphisms on $WA$. 
We also note that $pa := \underbrace{a + \cdots + a}_{\text{$p$ times}} = VF (a) = FV (a)$ holds for $a \in WA$.

For an $\mathbb F_p$-algebra $A$ and an ideal $I$ of $A$, 
we denote by $WI$ the kernel of the induced surjective ring homomorphism $WA \to W(A/I)$. 
As a set, we may write 
\[
WI = \{(a_0, a_1, a_2, \ldots)\in WA \mid a_i\in I\}.
\]
We may also define $W_n I$ as above.

The ring of Witt vectors can be generalized to the sheaf of Witt vectors 
(cf.\ \cite[Section 2]{BBE}, \cite[Section 1]{Illusie}). 
For a $k$-scheme $X$, 
we define the presheaf $W \mcO _X$ on $X$ as 
\[
(W \mcO _X) (U) := W(\mcO _X (U))
\]
for any open subset $U \subset X$. 
This presheaf is already a sheaf. 
For a quasi-coherent ideal sheaf $I \subset \mcO _X$, we can also define a sheaf $WI$ on $X$ similarly. 
Further, we can also define $W_n \mcO _X$ and $W_n I$ and $R, V, F$ similarly. 
We emphasize here that $W_n \mcO _X$, $W_n I$, $W \mcO_X$ and $W I$ do not have an $\mcO _X$-module 
structure if $n \not = 1$, but they are sheaves of abelian groups. 

For a $k$-scheme $X$, 
we can define a ringed space $W_n X$ as follows 
\[
W_n X := (|X|, W_n \mcO _X). 
\]
Since $W_n X = \Spec W_n \mcO _X$, it has a $W_n (k)$-scheme structure. 
If $f: X \to Y$ is a separated (resp.\ finite type, proper) morphism of $k$-schemes, 
then $W_n f: W_n X \to W_n Y$ is a separated (resp.\ finite type, proper) morphism of $W_n (k)$-schemes. 
We note here that for an ideal sheaf $I \subset \mcO _X$, 
the higher direct image sheaf $R^i (W_n f)_* (W_n I)$ coincides with $R^i f_* (W_n I)$, 
and we often use the latter notation. 

Lastly, we introduce the notation $WI_{\Q}$ and $R^i f_* (W I_{\mathbb{Q}})$ 
following \cite[Subsection 3.7]{CR2}. 
For a $k$-scheme $X$, we denote by $\mathcal{A}(X)$ the abelian category of 
the sheaves of abelian groups on $X$. 
We define 
\[
\mathcal{A}_{\text{b-tor}}(X) := \big\{ F \in \ob (\mathcal{A}(X)) \ 
\big| \ \exists n \in \mathbb{Z} _{>0} : n \cdot \text{id} _F = 0  \big\}
\]
as a full subcategory of $\mathcal{A}(X)$. 
Then the quotient category $\mathcal{A}(X)_{\Q} := \mathcal{A}(X)/\mathcal{A}_{\text{b-tor}}(X)$ exists and is 
an abelian category. 
We denote the projection functor by 
\[
q: \mathcal{A}(X) \to \mathcal{A}(X)_{\Q}. 
\]
We also use the notation $(-)_{\Q}:=q(-)$ for simplicity. 
For example, for a coherent ideal sheaf $I$ on $X$, $WI _{\mathbb{Q}}$ is defined by $WI _{\mathbb{Q}} = q(WI)$ and 
it is an object of $\mathcal{A}(X)_{\Q}$ (hence $WI _{\mathbb{Q}}$ is not a sheaf). 
For a morphism $f:X \to Y$ of $k$-schemes, 
the functor $f_* : \mathcal{A}(X) \to \mathcal{A}(Y)$ induces the functor on the quotient categories: 
\[
f_* : \mathcal{A}(X)_{\mathbb{Q}} \to \mathcal{A}(Y)_{\mathbb{Q}}, 
\]
and the right derived functor 
\[
R^i f_* : \mathcal{A}(X)_{\mathbb{Q}} \to \mathcal{A}(Y)_{\mathbb{Q}}
\]
for $i \ge 0$. 
When $Y = \Spec k$, we write $H^i (X, F_{\mathbb{Q}})$ instead of $R^i f_* F_{\Q}$ for a sheaf $F$ on $X$.

We give here some remarks on the notation $R^i f_* WI_{\Q}$. 
\begin{rmk}\label{cr-remark}
\begin{enumerate}
\item In \cite{CR2}, $R^i f_* WI_{\Q}$ is considered as a object in the category 
$\widehat{\text{dRW}}_{X, \Q} := \widehat{\text{dRW}}_{X}/\widehat{\text{dRW}}_{X, \text{b-tor}}$. 
However, for our aim, it is sufficient to work on $\mathcal{A}(X)_{\mathbb{Q}}$ described as above. 
\item 
Let $f:X \to Y$ be a morphism of separated schemes of finite type over $k$ and 
let $I$ be a coherent ideal sheaf on $X$. 
Then by \cite[Proposition 3.7.7]{CR2}, we have that 
$$R^if_*(WI_{\Q}) \simeq (R^if_*(WI))_{\Q}$$
for each non-negative integer $i$, in the category $\mathcal{A}(Y)_{\mathbb{Q}}$. 
\item 
Note that $R^i f_* (W I_{\mathbb{Q}})$ differs from the following sheaf: 
$$R^i f_* (W I) \otimes_{\mbZ} \mbQ: U \mapsto \Gamma(U, R^i f_* (W I)) \otimes_{\mbZ} \mbQ.$$ 
First of all, $R^i f_* (W I_{\mathbb{Q}})$ is not a sheaf on $Y$ (it is an object in $\mathcal{A}(Y)_{\mathbb{Q}}$). 
However, by (2) in this remark, the condition $R^i f_* (W I_{\mathbb{Q}})=0$ is equivalent to the condition that 
there exists a positive integer $N$ such that $N \cdot R^i f_* (W I)=0$. 
On the other hand, $R^i f_* (W I) \otimes_{\mbZ} \mbQ=0$ 
if and only if for any open subset $U$ of $Y$ and section $s \in \Gamma(U, R^i f_* (W I))$, 
there exists a positive integer $N_s$ such that $N_s \cdot s=0$. 
In particular, $R^i f_* (W I_{\mathbb{Q}})=0$ implies $R^i f_* (W I) \otimes_{\mbZ} \mbQ=0$. 
\item 
The notation $R^if_*(WI_K)$ used in \cite{BBE} is the same as $R^i f_* (W I) \otimes_{\mbZ} \mbQ$. 
However, some of their results (e.g.\ \cite[2.1, 2.2, 2.4]{BBE}) are still valid for $R^i f_* (W I_{\mathbb{Q}})$. 
See the proof of Proposition \ref{prop:BBE}. 
\item 
If $Z$ is a proper scheme over $k$, 
then $H^i(Z, W\mcO_{Z, \Q})=0$ if and only if $H^i(Z, W\mcO_Z) \otimes_{\mbZ} \Q=0$. 
Indeed, this equivalence follows from (3) in this remark and the fact that 
$H^i(Z, W\mcO_Z)$ is a finitely generated $W_{\sigma}[[V]]$-module (cf.\ the proof of Proposition 2.10 in \cite{BBE}). 
\end{enumerate}
\end{rmk}

\subsection{Fundamental properties of Witt vector cohomologies}\label{Sub-fundamental-Witt}
In this subsection, we summarize some basic properties of Witt vector cohomologies. 

\begin{lem}\label{higher-limit}
Let $f:X \to Y$ be a morphism 
between separated schemes of finite type over a perfect field $k$ of characteristic $p>0$. 
Let $I$ be a coherent ideal sheaf on $X$. 
Then, the following assertions hold. 
\begin{enumerate}
\item{$R^i\varprojlim_n(f_*W_nI)=0$ for every $i\geq 1$.}
\item{$R^i\varprojlim_n(R^jf_*W_nI)=0$ for every $i\geq 2$ and every $j\geq 0$.}
\end{enumerate}
\end{lem}

\begin{proof}
It is sufficient to show that the following two conditions hold for any affine open subset $U \subset X$ 
(cf.\ \cite[Lemma 1.5.1]{CR2}):
\begin{itemize}
\item $H^i(U, R^jf_*(W_nI))=0$ for any $i,n \ge 1$ and $j \ge 0$. 
\item $H^0 (U, f_*(W_{n+1}I)) \to H^0 (U, f_*(W_{n}I))$ is surjective for any $n \ge 1$. 
\end{itemize}
We note that $W_nI$ becomes a coherent ideal sheaf on a scheme $W_n X$ 
of finite type over an Artinian ring $W_n (k)$. 
Therefore, $R^jf_*(W_nI) = R^j (W_nf)_*(W_nI)$ is also a quasi-coherent sheaf on $W_n Y$. 
Since $R^jf_*(W_nI)$ is a quasi-coherent sheaf on the Noetherian scheme $W_n Y$, we see 
\[
H^i(U, R^jf_*(W_nI))=0
\]
for any affine open subset $U \subset X$ and $i>0$. 

The map $H^0 (U, f_*(W_{n+1}I)) \to H^0 (U, f_*(W_{n}I))$ is the same as 
\[
W_{n+1}(I(f^{-1}(U))) \to W_{n}(I(f^{-1}(U))), 
\]
and it is clearly surjective. 
\end{proof}

The following two lemmas are consequences of the above lemma. 
\begin{lem}\label{usual-Witt-higher}
Let $f:X \to Y$ be a morphism 
between separated schemes of finite type over a perfect field $k$ of characteristic $p>0$. 
Let $I$ be a coherent ideal sheaf on $X$ satisfying $R^if_*I=0$ for every $i>0$. 
Then $R^if_*(WI)=0$ holds for every $i>0$.
\end{lem}

\begin{proof}
First we note that $R^i f_* (W_n I) = 0$ holds for each $i \ge 1$ and $n \ge 1$. 
This follows from the induction on $n$ using the assumption and the short exact sequence of sheaves of abelian groups: 
\[
0 \to I \overset{V^n}\to W_{n+1}I \overset{R}\to W_nI \to 0. 
\]
Then the assertion follows from the following isomorphisms of complexes in the derived category
\begin{align*}
Rf_*(WI) &\cong Rf_* \big( R\varprojlim_n(W_nI)\big) \cong R\varprojlim_n \big( Rf_*(W_nI) \big) \\
&\cong R\varprojlim_n(f_*W_nI) \cong \varprojlim_n(f_*W_nI).
\end{align*}
The first and fourth isomorphisms follow from Lemma \ref{higher-limit} (1), and
the third follows from the above remark. 
For the second isomorphism, note that $f_* \circ \varprojlim_n=\varprojlim_n \circ f_*$ as a functor. 
\end{proof}

\begin{lem}\label{higher-fiber}
Let $f:X \to Y$ be a morphism 
between separated schemes of finite type over a perfect field $k$ of characteristic $p>0$. 
Let $I$ be a coherent ideal sheaf on $X$. 
Then, $R^if_*(WI)=0$ holds for every $i > m := \max_{y\in Y} \dim f^{-1}(y)$. 
\end{lem}

\begin{proof}
First we note that $R^if_* (W_nI) = 0$ holds for $i > m$ since 
the sheaf $W_nI$ is a quasi-coherent sheaf on the Noetherian scheme $W_n X$.

Since $R^{i}\varprojlim_n(W_nI)=0$ for $i>0$ (Lemma \ref{higher-limit} (1)), 
we obtain 
\[
Rf_*(WI)\cong Rf_*(R\varprojlim_nW_nI)\cong R(f_*\circ \varprojlim_n)(W_nI)
\cong R(\varprojlim_n \circ f_*)(W_nI) 
\]
as complexes in the derived category, and hence we have
$R^{i}f_*(WI) \cong R^{i}(\varprojlim_n \circ f_*)(W_nI)$ for each $i$. 
Thus, we obtain the following spectral sequence: 
\begin{align*}
E_2^{i, j} & := R^i \varprojlim_n R^jf_*(W_nI) \\
&\Longrightarrow 
E^{i+j} := R^{i+j}(\varprojlim_n \circ f_*)(W_nI) \cong R^{i+j}f_*(WI).
\end{align*}
By Lemma~\ref{higher-limit} (2), 
we see $E_2^{i, j}=0$ for $i\geq 2$. 
Then, for each $j$, 
we obtain the following exact sequence
\[
0 \to R^1\varprojlim_nR^{j-1}f_*(W_nI) \to R^{j}f_* (WI) \to 
\varprojlim_n R^{j}f_*(W_nI) \to 0.
\]
Therefore $R^{j}f_* (WI)=0$ holds if $j \ge m+2$. 

We show $R^{m+1}f_* (WI)=0$. 
By the above short exact sequence, we obtain
\[
R^1\varprojlim_nR^{m}f_*(W_nI) \cong R^{m+1}f_* (WI). 
\]
In order to show $R^1\varprojlim_nR^{m}f_*(W_nI) = 0$, 
it is sufficient to check the surjectivity of the map
\[
R^m f_* (W_{n+1}I) \to R^m f_* (W_nI). 
\]
It can be confirmed by the short exact sequence
\[
0 \to I \overset{V^n}\to W_{n+1}I \overset{R}\to W_nI \to 0, 
\]
and the vanishing $R^{m+1} f_* I = 0$. 
\end{proof}

The vanishing of $R^i f_* (WI)$ implies the vanishing of $R^i f_* (WI _{\mbQ})$. 
\begin{lem}\label{lem:Q}
Let $f:X \to Y$ be a morphism of separated schemes of finite type over a perfect field $k$ of characteristic $p>0$. 
Let $I$ be a coherent ideal sheaf on $X$ and let $i$ be a non-negative integer. 
If $R^i f_* (WI) = 0$ holds, then $R^i f_* (WI_{\mbQ}) = 0$ holds. 
\end{lem}
\begin{proof}
This easily follows from $R^i f_* (WI_{\mbQ}) \simeq R^i f_* (WI) _{\mbQ}$ 
by Remark \ref{cr-remark}.
\end{proof}

One of the advantages to consider the Witt vector cohomology is from the following lemma which states 
that the Witt vector cohomology does not change under a finite universal homeomorphism. 

\begin{lem}\label{geom-homeo}
Let $f:X \to Y$ be a proper surjective morphism of separated schemes of finite type over a perfect field 
$k$ of characteristic $p>0$. 
If any fiber of $f$ is geometrically connected, then 
the natural homomorphism $W \mcO_{Y, \Q} \to f_*W \mcO_{X, \Q}$ is an isomorphism. 
\end{lem}

\begin{proof}
Taking the Stein factorization of $f$, we may assume that $f$ is a finite universal homeomorphism.

We may assume that $X$ and $Y$ are reduced by \cite[Proposition 2.1 (i)]{BBE} (cf. Remark~\ref{cr-remark}). 
Moreover, we may assume that $X$ and $Y$ are affine and 
may write $X=\Spec B$ and $Y=\Spec A$. 
Since the induced ring homomorphism $\varphi:A \to B$ is injective, $W\varphi:WA \to WB$ is also injective. 

Since $f$ is a finite universal homeomorphism, 
$B^{p^e} \subset A$ holds for some $e\in\mathbb Z_{>0}$ (cf.\ \cite[Proposition 6.6]{Kollar:quot_space}), 
which implies $p^e(WB) \subset \varphi(WA)$. 
Thus $W\varphi_{\Q}:WA_{\Q} \to WB_{\Q}$ is an isomorphism. 
\end{proof}

The following proposition, which is a consequence of a theorem \cite[Theorem 2.4]{BBE} by Berthelot, Bloch and Esnault, is also a key property of the Witt vector cohomology. 
\begin{prop}\label{prop:BBE}
Let $f:X \to Y$ be a proper birational morphism of separated schemes of finite type 
over a perfect field $k$ of characteristic $p>0$, 
and let $I \subset \mcO _X$ be a coherent ideal sheaf on $Y$. 
If $f$ is an isomorphism outside $\Zeros (I) \subset X$, then $R^i f_* (W I_{\mbQ}) = 0$ holds for $i >0$. 
\end{prop}

\begin{proof}
Although $R^i f_* (W I_{\mbQ})$ differs from $R^i f_* (W I_{K})$ appearing in \cite{BBE} 
(cf. Remark~\ref{cr-remark}), 
we can apply the same argument as in \cite[Theorem 2.4]{BBE}. 
Indeed, it is shown in \cite[Theorem 2.4]{BBE} that there exists a positive integer $c$ such that
for any $i \ge 1$, the map $p^c: R^i f_* (W I) \to R^i f_* (W I)$ becomes zero. 
This shows that $R^i f_* (W I_{\mbQ}) \simeq R^i f_* (W I) _{\Q} = 0$ in the category $\mathcal{A}(Y) _{\Q}$. 
\end{proof}

For later use, we give two easy lemmas. 

\begin{lem}\label{zero-dim}
Let $k$ be a perfect field of characteristic $p>0$. 
Let $f:X \to Y$ be a finite morphism of zero-dimensional schemes 
which are of finite type over $k$. 
If the induced map 
$$H^0(Y, W\MO_Y)\otimes_{\mbZ} \mbQ  \to H^0(X, W\MO_X)\otimes_{\mbZ} \mbQ$$
is an isomorphism, then $f$ is a universal homeomorphism. 
\end{lem}

\begin{proof}
Since $W(A)\times W(B) \cong W(A\times B)$, 
we may assume that $\# Y = 1$. 
Taking the reduced scheme structures, 
we may assume that $X$ and $Y$ are reduced. 
We may also assume that $\# X =1$, since otherwise the induced map on the $0$-th Witt vector cohomologies 
is clearly not surjective. 
Therefore, we may write 
\[
X=\Spec k_X,\qquad Y=\Spec k_Y, 
\]
where $k \subset k_Y \subset k_X$ are finite extensions of fields. 
It suffices to show $k_X=k_Y$. 
If $k_Y \subsetneq k_X$, then $W(k_Y) \otimes_{\mbZ} \mbQ \subsetneq W(k_X) \otimes_{\mbZ} \mbQ$. 
This means that the induced map 
$$W(k_Y) \otimes_{\mbZ} \mbQ =H^0(Y, W\MO_Y) \otimes_{\mbZ} \mbQ \to
 H^0(X, W\MO_{X}) \otimes_{\mbZ} \mbQ= W(k_X) \otimes_{\mbZ} \mbQ$$
is not surjective, which leads to a contradiction. 
\end{proof}

\begin{lem}\label{zero-dim-higher}
Let $f:X \to Y$ be a continuous map between topological spaces. 
Let $A$ be a sheaf of abelian groups on $X$. 
Assume that there exists a finite subset $\{x_1, \cdots, x_r\}\subset X$ 
which satisfies the following properties. 
\begin{enumerate}
\item{For every $1\leq j \leq r$, $x_j$ is a closed point, that is, 
$\{x_j\}$ is a closed subset of $X$. }
\item{For every $1\leq j \leq r$, $f(x_j)$ is a closed point, that is, 
$\{f(x_j)\}$ is a closed subset of $Y$. }
\item{$A|_{X\setminus \{x_1, \cdots, x_r\}}=0$.}
\end{enumerate}
Then, $R^pf_*A=0$ for every $p>0$. 
\end{lem}

\begin{proof}
Since the condition (3) implies $i_*i^{-1}A \cong A$ where $i:\{x_1, \cdots, x_r\} \hookrightarrow X$, 
we can assume $X=\{x_1, \cdots, x_r\}$. 
By the condition (2), we can also assume that $f(X)$ is one point (set $\{y\} := f(X)$). 
Then, by the condition (1), 
for every sheaf of abelian groups $B$ on $X$ and every open subset $y\in Y' \subset Y$, 
we see 
\[
(f_*B)(Y')=B(f^{-1}(Y'))=B_{x_1}\oplus \cdots \oplus B_{x_r}.
\]
Thus, the functor $f_*$ is exact and hence $R^pf_*A=0$ for every $p>0$. 
\end{proof}

\section{$W\mcO$-rationality of klt threefolds}\label{S-vanishing}

For a variety $X$ admitting a resolution of singularities, 
we consider the following condition: 
\begin{enumerate}
\item[(a)] 
There exists a proper birational morphism $\varphi:V \to X$ 
from a smooth variety $V$ such that $R^i\varphi_*(W\MO_{V, \Q})=0$ holds for every $i>0$.
\end{enumerate}
A variety satisfying this property is said to have \textit{$W \mcO$-rational singularities} 
(see  \cite[Corollary 4.5.1]{CR2} and Remark \ref{cr-remark}). 
This property is stronger than the definition of Witt-rational singularity 
by Blickle and Esnault \cite{BE}, 
and weaker than that of the usual rational singularity (see \cite[Proposition 4.4.17]{CR2}). 
The original $W \mcO$-rational singularity was defined 
by Chatzistamatiou and R{\"u}lling and studied in \cite{CR2}. 
For instance, it is known that the condition (a) is independent of the choice of a resolution: 
\begin{rmk}\label{rmk:any}
When $X$ satisfies the condition (a), it is known that 
$$R^i\varphi_*(W\MO_{V, \Q})=0$$
holds 
for every $i>0$ and for {\bf any} resolution $\varphi$. 
More generally, it is known that the condition (a) is equivalent to 
the condition $R^i\varphi_*(W\MO_{V, \Q})=0$ holds for every $i>0$ for 
some (or any) quasi-resolution $\varphi: V \to X$ (See \cite[Corollary 4.5.1]{CR2}). 

We note that the existence of the quasi-resolution is known for varieties in any dimension 
by \cite[Corollary 5.15]{dJ:quasi-resolution}, and the $W \mcO$-rationality is defined using 
quasi-resolutions when we do not assume the existence of resolutions. 
We also note that for a proper birational morphism $\varphi: V \to X$ between smooth varieties 
the vanishing of the higher direct image sheaf $R^i \varphi _* \mcO _X$ for $i > 0$ 
is known by \cite[Theorem 2]{CR}. 
\end{rmk}
In this section, we show that the property (a) is preserved under pl-contractions and pl-flips 
in dimension three (Subsection \ref{Sub-Witt-pl-cont}). 
As a consequence, we see that an arbitrary klt threefold, of characteristic $p>5$, 
has $W \mcO$-rational singularities (Theorem \ref{thm:WO}). 
For a characteristic-independent argument, 
we also discuss the bijectivity of the normalization of plt centers
in Subsection \ref{Sub-bije-pl-cont}. 
As a consequence in characteristic $p \leq 5$, 
we reduce the problem on the existence of flips to 
a certain special cases (Theorem \ref{thm:235}). 
The readers who are not interested in the low characteristic 
case may skip Subsection \ref{Sub-bije-pl-cont} (see also Remark \ref{rmk:assumption}).

\subsection{$W\mcO$-rationality under pl-contractions and pl-flips}\label{Sub-Witt-pl-cont}
In this subsection, 
we prove Proposition \ref{pl-cont-Witt} and Proposition \ref{flip-Witt-vanish}, 
which state that the $W \mcO$-rationality is preserved under pl-contractions and pl-flips. 

First we recall the definition of pl-contractions and pl-flips. 
\begin{defi}
Let $(X, \Delta)$ be a dlt pair and 
let $S=\sum_{i\in I}S_i$ be the irreducible decomposition of $S:= \llcorner \Delta \lrcorner$. 
We say a proper birational morphism $f:X \to Y$ 
is a \textit{pl-divisorial} (resp.\ \textit{pl-flipping}) contraction if 
the following conditions hold. 
\begin{itemize}
\item $f_*\MO_X=\MO_Y$.
\item $\rho(X/Y)=1$.
\item $-(K_X+\Delta)$ is $f$-ample. 
\item $-S_0$ is $f$-ample for some $0\in I$. 
\item $\dim\Ex(f)=\dim X-1$ (resp.\ $\dim\Ex(f)<\dim X-1$). 
\end{itemize}

For a pl-flipping contraction $f: X \to Y$, 
a proper birational morphism $f^+: X^+ \to Y$ is called a \textit{pl-flip} if the following conditions hold
\begin{itemize}
\item $f^+ _*\MO_{X^+}=\MO_Y$.
\item $\rho(X^+ /Y)=1$.
\item $K_{X^+}+ \Delta ^+$ is $\mbQ$-Cartier and $f^+$-ample, 
where $\Delta ^+$ is the strict transform of $\Delta$ on $X^+$. 
\item $\dim \Ex(f^+) < \dim X-1$. 
\end{itemize}
\end{defi}

\noindent
The advantage of considering pl-contractions is that some problems on the threefold $X$ 
can be reduced to problems on the surfaces $S_i$. 
The following result on surfaces plays a crucial role 
in the proof of Proposition \ref{pl-cont-Witt}. 

\begin{prop}\label{surface-Witt}
Let $(S, \Delta_S)$ be a dlt surface over a perfect field $k$ 
of characteristic $p > 0$ and 
let $f:S \to U$ be a projective morphism to 
a separated scheme $U$ of finite type over $k$. 
If $-(K_S+\Delta_S)$ is $f$-ample, then $R^if_*\MO_S=0$ holds for all $i>0$. 
In particular, $R^if_*(W\MO_{S, \Q})=0$ holds for every $i>0$.
\end{prop}

\begin{proof}
By Lemma \ref{usual-Witt-higher} and Lemma \ref{lem:Q}, 
it suffices to show that $R^if_*\MO_S=0$ for $i>0$. 
Since the base change $(X \times_k \overline k, \Delta \times_k \overline k)$ 
to the algebraic closure is also dlt, we may assume that $k$ is an algebraically closed field. 
Since dlt surfaces are always $\mbQ$-factorial (cf.\ \cite[Proposition 6.3]{FT}), 
we may assume that $(S, \Delta_S)$ is klt. 
We may also assume that $f_*\MO_S=\MO_U$ by the Stein factorization. 

If $\dim U \geq 1$, then the assertion follows from \cite[Theorem 2.12]{T1}. 
Suppose $\dim U=0$. 
In this case, the assertion follows from the fact that 
$S$ is a rational surface (\cite[Theorem 3.5]{T1}) with rational singularities 
(cf.\ \cite[Theorem 6.4]{FT}). 
\end{proof}

We prove that the $W \mcO$-rationality is preserved under pl-contractions. 
\begin{prop}\label{pl-cont-Witt}
Let $k$ be a perfect field of characteristic $p > 0$. 
Let $(X, \Delta)$ be a dlt threefold over $k$ and 
let $S=\sum_{i=0}^r S_i$ be the irreducible decomposition of $S:= \llcorner \Delta \lrcorner$. 
Let $f:X \to Y$ be a projective birational morphism 
with the following properties. 
\begin{itemize}
\item{$f_*\MO_X=\MO_Y$.}
\item{$-(K_X+\Delta)$ is $f$-ample.}
\item{$-S_0$ is $f$-ample.}
\item{The normalization $\nu:S_0^N \to S_0$ is a universal homeomorphism.}
\end{itemize}
Then, the following assertions hold. 
\begin{enumerate}
\item $R^if_*(W\MO_{X, \Q})=0$ for every $i>0$. 
\item If $X$ has $W \mcO$-rational singularities, then so does $Y$. 
\end{enumerate}
\end{prop}

\begin{rmk}\label{rmk:assumption}
The fourth assumption is known to be true when $(X, \Delta)$ is 
$\mbQ$-factorial dlt threefold of characteristic $p>5$. 
Indeed, in this case, $S_0$ is known to be normal \cite[Proposition 4.1]{HX} 
(cf.\ Theorem \ref{thm:normality}). 
In the next subsection, this condition is called ``the condition (b)" and 
we study its behavior under pl-contractions and pl-flips in arbitrary positive characteristic. 
\end{rmk}

\begin{proof}[Proof of Proposition \ref{pl-cont-Witt}]
The assertion (2) follows from (1) by the Grothendieck spectral sequence. 
In what follows, we prove the assertion (1). 
Since $-S_0$ is $f$-ample, we see $\Ex(f)\subset S_0$. 
Therefore, by Proposition \ref{prop:BBE}, it follows that 
\[
R^if_*(WI_{S_0, \Q})=0
\]
for any $i > 0$, where $I_{S_0} \subset \mcO _X$ is the defining ideal sheaf of $S_0 \subset X$. 
By the following exact sequence: 
\[
0 \to WI_{S_0, \Q} \to W\MO_{X, \Q} \to \iota_*W\MO_{S_0, \Q} \to 0,
\]
where $\iota:S_0 \hookrightarrow X$, 
it follows that 
\[
R^if_*(W\MO_{X, \Q}) \cong R^if_*(\iota_*W\MO_{S_0, \Q})
\]
for $i > 0$. 
Since $\nu: S_0^N \to S_0$ is a universal homeomorphism, it follows that 
$\nu_*W\MO_{S_0^N} \cong W\MO_{S_0, \Q}$ by Lemma \ref{geom-homeo}, and we obtain 
\[
R^if_*(\iota_*W\MO_{S_0, \Q}) \cong R^if_*(\iota_*\nu_*W\MO_{S^N_0, \Q}) 
\cong R^i(f \circ \iota \circ \nu)_*(W\MO_{S^N_0, \Q}).
\]
Note that the latter isomorphism follows from 
$R^i (\iota \circ \nu) _* (W\MO_{S^N_0, \Q}) = 0$ for $i > 0$ by Lemma \ref{higher-fiber}. 
Since the pair $(S^N_0, \Delta _{S^N_0})$, defined by 
$(K_X + \Delta) |_{S^N_0} = K_{S^N_0} + \Delta _{S^N_0}$, is dlt by adjunction, 
the last sheaf is zero by Proposition \ref{surface-Witt}, which completes the proof. 
\end{proof}

Next, we show Proposition \ref{flip-Witt-vanish}. 
We start with the following result on surfaces over a (possibly imperfect) field (see Remark \ref{rmk:imperfect}). 

\begin{lem}\label{imperfect}
Let $f:T \to S$ be a proper birational morphism of two-dimensional excellent schemes, 
where $T$ is regular and $(S, \Delta_S)$ is dlt for some effective $\mbQ$-divisor $\Delta _S$ on $S$. 
Then, $R^if_*\MO_T=0$ holds for $i>0$. 
\end{lem}

\begin{proof}
We may replace $f$ with the minimal resolution of $S$. 
Note that the minimal resolutions 
exist by \cite[Theorem 10.5]{Kollar}. 
Then, the $\Q$-divisor $\Delta_T$, defined by $K_T+\Delta_T=f^*(K_S+\Delta_S)$, 
is effective by the negativity lemma. 
If $E$ is an $f$-exceptional curve on $T$, 
then the image $f(E)$ is contained in the non-regular locus of $S$, 
and hence $a_E(S, \Delta_S) <0$. 
Since $(S, \Delta_S)$ is dlt, we obtain $a_E(S, \Delta_S)>-1$. 
This means that, if $f$ is not an isomorphism, then 
some coefficient of $\Delta_T$ is less than one. 
Therefore, $R^1 f_* \mcO _T=0$ follows from \cite[Theorem 10.4]{Kollar} by letting 
\[
L:=\MO_T,\quad N:=(m-1)(K_T+\Delta_T)-f^*(m(K_S+\Delta_S)), 
\] 
where $m$ is a positive integer such that $m(K_S+\Delta_S)$ is Cartier. 
\end{proof}

Lemma \ref{imperfect} provides the following consequence on threefolds. 
\begin{lem}\label{3fold-rational}
Let $(X, \Delta)$ be a dlt threefold over a perfect field $k$ of characteristic $p > 0$. 
Let $\varphi: V \to X$ be a proper birational morphism from a smooth threefold $V$. 
Then, the following assertions hold. 
\begin{enumerate}
\item There exists a zero-dimensional closed subset $Z$ of $X$ 
such that $$(R^i\varphi_*\MO_V)|_{X\setminus Z}=0$$ for any $i>0$. 
\item There exists a zero-dimensional closed subset $Z$ of $X$ 
such that $$(R^i\varphi_*(W\MO_{V}))|_{X\setminus Z}=0$$ for any $i>0$. 
\end{enumerate}
\end{lem}

\begin{proof}
The assertion (2) follows from (1) by Lemma \ref{usual-Witt-higher}. 
In what follows, we prove the assertion (1). 
Suppose that $\Supp(R^i\varphi_*\MO_V)$ contains a curve $C$ for some $i>0$ and we derive a contradiction. 
Let $\xi_C$ be the generic point of $C$ and consider the following square of a fiber product: 
\[\xymatrix{
V \ar[d]^{\varphi} & V\times_X \Spec \MO_{X, \xi_C}=:W \ar[l]_{\alpha \hspace{17mm}} \ar[d]^{\psi} \\
X & \Spec \MO_{X, \xi_C} \ar[l]_{\beta \hspace{5mm}}
}\]
By the flat base change theorem, it follows that
\[
0\neq \beta^*R^i\varphi_*\MO_V \cong R^i\psi_*(\alpha^*\MO_V) \cong R^i\psi_*\MO_W.
\]
Note that $(\Spec \MO_{X, \xi_C}, \Delta|_{\Spec \MO_{X, \xi_C}})$ 
is a two dimensional excellent scheme and that $W$ is regular. 
Hence, we obtain a contradiction by Lemma \ref{imperfect}. 
\end{proof}

The $W\mcO$-rationality is preserved under pl-flips. 
\begin{prop}\label{flip-Witt-vanish}
Let $k$ be a perfect field of characteristic $p > 0$, and let 
$g:Y \to Z$ be a proper birational morphism between normal threefolds over $k$ such that 
any fiber of $g$ is at most one dimensional. 
Assume that $(Y, \Delta_Y)$ is dlt for some effective $\Q$-divisor $\Delta_Y$. 
If $Z$ has $W \mcO$-rational singularities, then so does $Y$. 
\end{prop}

\begin{proof}
Let $\psi: V \to Y$ be a resolution of $Y$ (cf.\ \cite{CP}). 
We prove $R^i\psi_*(W\MO_{V, \Q})=0$ for $i>0$. 
We denote by $\varphi$ the composition $g \circ \psi$: 
\[
\varphi:V \overset{\psi}\to Y \overset{g}\to Z. 
\]
Consider the Grothendieck spectral sequence: 
\[
E_2^{i, j}=R^ig_*R^j\psi_*(W\MO_{V, \Q}) \Longrightarrow E^{i+j} = R^{i+j}\varphi_*(W\MO_{V, \Q}).
\]
Here, we claim that $E_2^{i, j} =0$ holds for $i>0$. 
\begin{claim}\label{claim:E_2^{i, j}}
$E_2^{i, j} =R^ig_*R^j\psi_*(W\MO_{V, \Q}) =0$ holds for every $i>0$ and $j \ge 0$. 
\end{claim}
\noindent
First we assume this claim and finish the proof. 
This claim and the spectral sequence above yield 
\[
g_* \big(R^j\psi_*(W\MO_{V, \Q}) \big) \cong R^{j}\varphi_*(W\MO_{V, \Q}), 
\]
and this is zero for $j > 0$ by the $W\mcO$-rationality of $Z$ (cf.\ Remark \ref{rmk:any}). 
On the other hand, by Lemma \ref{3fold-rational}, 
the support of $R^j\psi_*(W\MO_{V})$ is at most a finite set for $j > 0$. 
Therefore, $g_*R^j\psi_*(W\MO_{V, \Q})=0$ implies $R^j\psi_*(W\MO_{V, \Q})=0$, 
which shows the $W\mcO$-rationality of $Y$. 

\begin{proof}[Proof of Claim \ref{claim:E_2^{i, j}}]
We divide the proof into three cases: (I) $i>0, j>0$, (II) $i=1$, $j=0$, and (III) $i \geq 2$, $j=0$. 

(I) 
Suppose that $i>0$ and $j>0$. 
By Lemma \ref{3fold-rational}, 
$\Supp \big( R^j\psi_*(W\MO_{V}) \big)$ is at most a finite set. 
Therefore, $R^ig_*R^j\psi_*(W\MO_{V})=0$ holds by Lemma~\ref{zero-dim-higher}. 
In particular, we get $R^ig_*R^j\psi_*(W\MO_{V, \Q})=0$.

(II) 
Suppose that $i=1$ and $j=0$. 
Then $E^{1,0} _2 = 0$ follows from 
the injectivity $E_2^{1, 0} \hookrightarrow E^1$ and the vanishing $E^1=R^1\varphi_*(W\MO_{V, \Q})=0$ 
which follows from the $W \mcO$-rationality of $Z$. 

(III) 
Suppose that $i \ge 2$ and $j=0$. 
Then by Lemma~\ref{higher-fiber}, 
\[
E_2^{i, 0}=R^ig_*(\psi_*(W\MO_{V, \Q}))=R^ig_*(W\MO_{Y, \Q})
\]
turns out to be zero since any fiber of $g$ is at most one dimensional. 
\end{proof}
This completes the proof of  Proposition \ref{flip-Witt-vanish}. 
\end{proof}

\subsection{Normalization for plt centers under pl-contractions and pl-flips}\label{Sub-bije-pl-cont}

For a $\Q$-factorial dlt threefold $(Y, \Delta)$, we consider the following condition: 
\begin{enumerate}
\item[(b)] The normalization of each irreducible component of $\llcorner \Delta\lrcorner$ is 
a universal homeomorphism. 
\end{enumerate}
In this subsection, we study its behavior under pl-contractions (Proposition \ref{pl-cont-bije}) 
and pl-flips (Proposition \ref{pl-flip-bije}). 

First we prove the following lemma, 
which is a consequence of the Koll\'ar--Shokurov connectedness lemma for surfaces 
(\cite[Proposition 2.3]{T3}). 

\begin{lem}\label{Kollar-Shokurov}
Let $(S, \Delta_S)$ be a dlt surface over a perfect field of characteristic $p>0$ and 
let $\pi:S \to U$ be a projective surjective morphism 
such that every fiber of $\pi$ is geometrically connected. 
Assume that $-(K_S+\Delta_S)$ is $\pi$-ample. 
Let $0\leq D \leq \llcorner \Delta_S \lrcorner$ be a divisor satisfying $D= \llcorner D\lrcorner$. 
Then, $\Supp(D) \cap \pi^{-1}(u)$ is geometrically connected for every closed point $u\in U$. 
\end{lem}

\begin{proof}
Replacing $k$ by the algebraic closure, we may assume that $k$ is algebraically closed. 
Taking the Stein factorization of $\pi$, 
we can assume that $\pi_*\MO_S=\MO_U$. 
Note that $S$ is $\Q$-factorial (cf.\ \cite[Proposition 6.3]{FT}). 
By taking a small perturbation, 
we can find $0\leq \Delta_S' \leq \Delta_S$ 
such that $D=\llcorner \Delta_S' \lrcorner$ and that $-(K_S+\Delta_S')$ is $\pi$-ample. 
Then, the assertion follows from \cite[Proposition 2.2]{T3}. 
\end{proof}

The condition (b) is preserved under pl-contractions. 
\begin{prop}\label{pl-cont-bije}
Let $k$ be a perfect field of characteristic $p > 0$. 
Let $(X, \Delta)$ be a $\mbQ$-factorial  dlt threefold over $k$ and 
let $S=\sum_{i=0}^r S_i$ be the irreducible decomposition of $S:= \llcorner \Delta \lrcorner$. 
Let $f:X \to Y$ be a projective birational morphism 
with the following properties. 
\begin{itemize}
\item $f_*\MO_X=\MO_Y$. 
\item $-(K_X+\Delta)$ is $f$-ample. 
\item $-S_0$ is $f$-ample. 
\item For every $i$, the normalization morphism $S_i^N \to S_i$ is a universal homeomorphism. 
\end{itemize}
Let $T_i:=f(S_i)$ be the scheme-theoretic image. 
Then for every $i\in I$, the normalization morphism $T_i^N \to T_i$ is a universal homeomorphism. 
\end{prop}

\begin{proof}
For each $i$, 
it is sufficient to show that every fiber of $g_i:=f|_{S_i} :S_i \to T_i$ is 
geometrically connected.

Suppose $i = 0$. 
Since $-S_0$ is $f$-ample, we see $\Ex(f) \subset S_0$. 
As $f$ has geometrically connected fibers, 
every fiber of $g_0:S_0 \to T_0$ is also geometrically connected. 

Suppose $i \not = 0$. 
Let $y \in T_i \subset Y$. 
It is enough to show that $g_i^{-1}(y)$ is geometrically connected. 
We can assume that $y \in f(\Ex(f))$, in particular, $y \in T_0$. 
It follows that  
\[
g_i^{-1}(y)=S_0 \cap (f|_{S_i})^{-1}(y)=S_0 \cap S_i \cap f^{-1}(y) 
=S_i \cap g_0^{-1}(y).
\]
Hence it is sufficient to show that 
each fiber of the composition 
\[
C_i := S_0 \cap S_i \hookrightarrow S_0 \xrightarrow{g_0} T_0
\]
is either empty or geometrically connected. 
Now, we set $\nu_0:S_0^N \to S_0$ and 
define $\Delta_{S_0^N}$ by $(K_X+\Delta)|_{S_0^N}=K_{S_0^N}+\Delta_{S_0^N}$. 
Note that the intersection $C_i := S_0 \cap S_i$ is either empty or 
purely one-dimensional because $X$ is $\Q$-factorial. 
Thus, $\nu_0^{-1}(C_i)$ is also purely one-dimensional and 
$\nu_0^{-1}(C_i) \subset \Supp(\llcorner \Delta_{S_0^N}\lrcorner)$ holds by adjunction. 
Further, $(S_0^N, \Delta _{S_0^N})$ is dlt and $-(K_{S_0^N}+\Delta_{S_0^N})$ is 
$(g_0 \circ \nu)$-ample. 
Then, by Lemma \ref{Kollar-Shokurov}, 
it follows that 
each fiber of the composition 
\[
\nu_0^{-1}(C_i) \hookrightarrow S_0 ^N \xrightarrow{g_0 \circ \nu} T_0
\]
is geometrically connected, which is what we want to show. 
\end{proof}

Condition (b) is preserved under pl-flips if we make an additional assumption 
on the Witt vector cohomology.

\begin{prop}\label{pl-flip-bije}
Let $g:Y \to Z$ be a proper birational morphism between normal threefolds over a perfect field of characteristic $p>0$, 
such that any fiber of $g$ is at most one dimensional. 
Let $S_Y$ be a prime divisor on $Y$ 
and set $S_Z:=g_*(S_Y)$. 
Assume the following three conditions. 
\begin{itemize}
\item $S_Y$ is regular in codimension one.
\item The normalization $S_Z^N \to S_Z$ is a universal homeomorphism. 
\item $R^1g_*(W\MO_{Y, \Q})=0$. 
\end{itemize}
Then the normalization $\nu:S_Y^N \to S_Y$ is a universal homeomorphism. 
\end{prop}

\begin{proof}
Set $g|_{S_Y}=:h:S_Y \to S_Z$. 
First we prove $R^1h_*(W\MO_{S_Y, \Q})=0$. 
Consider the following exact sequence: 
\[
0 \to WI_{S_Y, \Q} \to W\MO_{Y, \Q} \to W\MO_{S_Y, \Q} \to 0.
\]
Since $\dim \Ex(g) \leq 1$, it follows that $R^2 g_* (WI_{S_Y, \Q}) = 0$ by Lemma \ref{higher-fiber}. 
Thus, $R^1h_*(W\MO_{S_Y, \Q})=0$ follows from the assumption $R^1g_*(W\MO_{Y, \Q})=0$. 

Let $C \subset S_Y$ be the conductor of the normalization $\nu:S_Y^N \to S_Y$ 
and let $D:=\nu^{-1}(C)\subset S_Y^N$ be its scheme-theoretic inverse image. 
Then we have the following exact sequence: 
\[
0 \to W\MO_{S_Y, \Q} \to \nu_*W\MO_{S_Y^N, \Q} \oplus W\MO_{C, \Q} \to \nu_*W\MO_{D, \Q} \to 0. 
\]
Since $R^1h_*(W\MO_{S_Y, \Q})=0$, we obtain the exact sequence
\[
0 \to h_*W\MO_{S_Y, \Q} \to h_*\nu_*W\MO_{S_Y^N, \Q} \oplus h_*W\MO_{C, \Q} 
\to h_*\nu_*W\MO_{D, \Q} \to 0.
\]
Note that $S_Y^N \to S_Z^N$ is a birational morphism with geometrically connected fibers 
by the Zariski main theorem. 
Since the normalization $S_Z^N \to S_Z$ is a universal homeomorphism, 
both of $h\circ \nu:S_Y^N \to S_Z$ and $h$ have geometrically connected fibers, hence we obtain 
\[
h_*W\MO_{S_Y, \Q} \cong W\MO_{S_Z, \Q} \cong h_*\nu_*W\MO_{S_Y^N, \Q}
\]
by Lemma \ref{geom-homeo}. 
Thus, we obtain $h_*W\MO_{C, \Q} \cong h_*\nu_*W\MO_{D, \Q}$ by the exact sequence above, and hence
\[
H^0(C, W\MO_{C, \Q}) \cong H^0(D, W\MO_{D, \Q}).
\]
Note that $C$ and $D$ are zero-dimensional by the first assumption. 
Therefore, by Lemma \ref{zero-dim}, it follows that 
$C \to D$ is a universal homeomorphism,  
hence so is $\nu:S_Y^N \to S_Y$. 
\end{proof}

\subsection{$W \mcO$-rationality of klt threefolds}\label{Sub-proof}
In this subsection, we prove that the conditions (a) and (b) hold 
for $\mbQ$-factorial dlt threefolds modulo the following conjecture on the existence of flips. 

\begin{conj}\label{conj-flip}
Let $k$ be a perfect field of characteristic $p>0$. 
Let $f:(Y, T+B) \to Z$ be a pl-flipping contraction from a $\Q$-factorial plt threefold $(Y, T+B)$
over $k$, 
where $\llcorner T+B\lrcorner=T$. 
If the normalization of $T$ is a universal homeomorphism, then a flip of $f$ exists. 
\end{conj}

\noindent
We note that Conjecture \ref{conj-flip} is known for $p >5$. 
More generally, the existence of flips is known for $\mbQ$-factorial dlt threefolds over 
a perfect field of characteristic $p > 5$ (cf.\ Theorem \ref{thm:flip}).

\begin{thm}\label{main-normal}
Fix a perfect field $k$ of characteristic $p>0$ and 
assume that Conjecture \ref{conj-flip} holds for $k$. 

Let $(X, \Delta)$ be a $\Q$-factorial dlt threefold over $k$. 
Then the following assertions hold. 
\begin{enumerate}
\item $X$ has $W \mcO$-rational singularities. 
\item For any irreducible component $S$ of $\llcorner \Delta\lrcorner$, 
its normalization morphism $\nu:S^N \to S$ is a universal homeomorphism. 
\end{enumerate}
\end{thm}

\begin{proof}
Since $X$ is $\Q$-factorial, 
we may assume that $\Delta$ is a prime divisor for the proof of (2), and $\Delta = 0$ for (1). 
Set $S := \Delta$. 

Let $f:Y \to X$ be a log resolution of $(X, S)$. 
Then we may write 
\[
K_Y+S_Y+E=f^*(K_X+S)+E' \sim _{\mbQ, X} E', 
\]
where $S_Y$ is the strict transform of $S$, $E$ is the sum of all the $f$-exceptional prime divisors, and 
$E'$ is an effective $f$-exceptional $\Q$-divisor. 
Note that $\Ex(f)=\Supp E'$ as $(X, \Delta)$ is plt. 
We shall inductively construct a $(K_Y+S_Y+E)$-MMP over $X$ 
\[
(Y, S_Y+E)=:(Y_0, S_0+E_0) \overset{g_0}\dashrightarrow (Y_1, S_1+E_1) 
\overset{g_1}\dashrightarrow \cdots, 
\]
with the following two conditions:
\begin{enumerate}
\item[(a)$_j$] $Y_j$ has $W \mcO$-rational singularities. 
\item[(b)$_j$] The normalization of each irreducible component of $S_j+E_j$ is 
a universal homeomorphism. 
\end{enumerate}
Further, we shall prove that this MMP terminates and ends with $X$ itself. 
We note that the conditions (a)$_0$ and (b)$_0$ are satisfied since $(Y_0, S_0+E_0)$ is log smooth. 

Suppose that we have already constructed $(Y_j, S_j+E_j)$ with the conditions (a)$_j$ and (b)$_j$. 
If $K_{Y_j}+S_j+E_j$ is nef over $X$, then $Y_j = X$ holds by 
the negativity lemma and the $\Q$-factoriality of $X$. 
Suppose that $K_{Y_j}+S_j+E_j$ is not nef over $X$. 
Since $K_{Y_j}+S_j+E_j$ is pseudo-effective over $X$, by Lemma \ref{lemma-cone}, 
we can find a $(K_{Y_j}+S_j+E_j)$-negative extremal ray $R$ of $\overline{NE}(Y_i/X)$. 
Since the condition (b)$_j$ holds and $F \cdot R < 0$ holds for some component $F$ of $E_j$, 
the assumptions in Lemma \ref{lemma-cont} are satisfied. 
Hence, we obtain a projective birational morphism $h:Y_j \to Z$ over $X$ corresponding to $R$. 
Then Proposition \ref{pl-cont-Witt} and Proposition \ref{pl-cont-bije} ensure that 
the new pair $(Z, h_* S_j +h_* E_j )$ satisfies the conditions (a) and (b). 
Hence, if $h$ is a divisorial contraction, 
it is sufficient to set
\[
Y_{j+1}:=Z, \quad S_{j+1}:=h_* S_j, \quad E_{j+1}:=h_* E_j
\]
for our purpose. 
Suppose that $h$ is a flipping contraction. 
Since $(Y_j, S_j+E_j)$ satisfies the condition (b)$_j$, we may apply Conjecture \ref{conj-flip}, 
and a flip $h^+:Y_{j+1} \to Z$ of $h$ exists. 
Set $S_{j+1}$ and $E_{j+1}$ to be the proper transforms on $Y_{j+1}$ of $S_j$ and $E_j$, respectively. 
Then, the new pair $(Y_{j+1}, S_{j+1}+E_{j+1})$ satisfies 
the condition (a)$_{j+1}$ by Proposition \ref{flip-Witt-vanish}. 
Further, by Lemma \ref{R1}, each irreducible component of $S_{j+1}+E_{j+1}$ is regular in codimension one. 
Thus we may apply Proposition \ref{pl-flip-bije} and the condition (b)$_{j+1}$ is also satisfied. 

We note that this MMP terminates by the special termination (Theorem \ref{special termination}).
Since this MMP ends with $X$, the assertions (1) and (2) follow 
from the conditions (a)$_j$ and (b)$_j$ respectively. 
\end{proof}

In characteristic $p\leq 5$, we obtain the following reduction result 
by the same proof as in \cite[Theorem 4.3.7]{Fujino:ST}. 

\begin{thm}\label{thm:235}
If Conjecture \ref{conj-flip} holds, then a flip of every flipping contraction exists. 
In other words, when we show the existence of pl-flips for $\Q$-factorial plt threefolds $(X, S+B)$ 
with $\llcorner S+B\lrcorner=S$, 
we may assume that the normalization of $S$ is a universal homeomorphism. 
\end{thm}

In characteristic $p>5$, any klt threefold 
(which is not necessarily $\mathbb{Q}$-factorial) has $W\MO$-rational singularities. 

\begin{thm}\label{thm:WO}
Let $(X, \Delta)$ be a klt threefold over a perfect field of characteristic $p>5$. 
Then $X$ has $W \mcO$-rational singularities. 
\end{thm}

\begin{proof}
The assertion is proved in Theorem~\ref{main-normal} when $X$ is $\Q$-factorial. 
Thus it suffices to find a proper birational morphism $g:Y \to X$ 
from a $\Q$-factorial klt threefold satisfying $R^ig_*W\MO_{Y, \Q}=0$ for $i >0$. 
Since the problem is local and the non-$\mbQ$-factorial points on $X$ are isolated, 
we may assume that $X$ contains the unique non-$\mbQ$-factorial point $z$. 

By Proposition~\ref{MFS-special}, 
we can find a projective birational morphism $h:Y \to X$ with the following conditions:
\begin{itemize}
\item $Y$ is $\mbQ$-factorial and klt, 
\item $h$ is isomorphic over $X \setminus \{ z \}$. 
\item $E := h^{-1}(z)_{\text{red}}$ is a surface of Fano type. 
\end{itemize}

\noindent
By the following short exact sequence
\[
0 \to W I_{E, \mbQ} \to W \mcO _{Y, \mbQ} \to W \mcO _{E, \mbQ} \to 0 
\]
and the vanishing $R ^i h_* (W I_{E, \mbQ}) = 0$ for $i > 0$ (Proposition \ref{prop:BBE}), 
we obtain 
\[
R^ih_*(W\MO_{Y, \Q}) \cong R^ih_*(W\MO_{E, \Q})
\]
for $i > 0$. 
The latter sheaf $R^ih_*(W\MO_{E, \Q})$ is zero by Proposition~\ref{surface-Witt}. 
\end{proof}

\subsection{Examples}\label{subsection:non-rational_klt}
In this subsection, we give two examples on rational and $W\mcO$-rational singularities. 
Before we construct the first example, we show an auxiliary lemma. 

\begin{lem}\label{lemma-homeo-WO}
Let $f:X \to Y$ be a proper birational morphism of normal surfaces 
over an algebraically closed field of characteristic $p>0$. 
Assume that $E:={\rm Ex}(f)$ is a rational curve whose normalization is a universal homeomorphism. 
Then $R^if_*(W\MO_{X, \mathbb{Q}})=0$ for $i>0$. 
\end{lem}

\begin{proof}
By Proposition \ref{prop:BBE}, 
we obtain 
\[
R^if_*(W\MO_{X, \mathbb{Q}}) \cong H^i(E, W\MO_{E, \mathbb{Q}}). 
\]
Let $\nu:\mathbb P^1 \to E$ be the normalization. 
It follows from Lemma \ref{geom-homeo} 
that $\nu_*W\MO_{\mathbb P^1, \mathbb{Q}}=W\MO_{E, \mathbb{Q}}$. 
Therefore, by Lemma \ref{usual-Witt-higher} and Lemma \ref{lem:Q}, we obtain 
\[
H^i(E, W\MO_{E, \mathbb{Q}}) \cong 
H^i(\mathbb P^1, W\MO_{\mathbb P^1, \mathbb{Q}})
= 0
\]
for $i>0$. 
\end{proof}

The first example is a two-dimensional $W\mcO$-rational singularity 
which is not a rational singularity (cf.\ \cite[Section~4.6]{CR2}). 

\begin{example}
Let $f:X \to Y$ be a proper birational morphism from a smooth surface to a normal surface 
over an algebraically closed field of characteristic $p>0$
such that $\Ex(f)$ is an irreducible cuspidal cubic curve. 
Then, the singularity of $Y$ is $W\MO$-rational by Lemma~\ref{lemma-homeo-WO}, 
but not a rational singularity. 
\end{example}

Next, we give an example which is klt but with non-rational singularities. 
We give an example in characteristic two. 
It is based on an example of weak del Pezzo surfaces $X$ with 
$H^1(X, \MO_X) \not = 0$ 
which was given by Schr{\"o}er 
(Theorem 8.1, 8.2 and Section 9 in \cite{schroer}. See also \cite{maddock}).

\begin{prop}\label{ex:dim4}
The following assertions hold. 
\begin{enumerate}
\item{For an arbitrary imperfect field $k$ of characteristic two, 
there exists a three-dimensional klt pair $(Z, \Delta_Z)$ over $k$ 
such that $R^1g_*\MO_{Y} \neq 0$ for a resolution $g:Y \to Z$ and $Z$ is the cone of some 
log canonical pair $(X, \Delta)$.}
\item{For an arbitrary perfect field $k$ of characteristic two, 
there exists a four-dimensional klt pair $(Z', \Delta_Z')$ over $k$ 
such that $R^1g'_*\MO_{Y'} \neq 0$ for a resolution $g':Y' \to Z'$. }
\end{enumerate}
\end{prop}

\begin{proof}
Since (1) implies (2), we only show (1). 
Let $k$ be an arbitrary imperfect field of characteristic two. 
By \cite[Section~9]{schroer}, there exists a regular projective surface $X$ over $k$ 
such that $-K_X$ is nef and big and $H^1(X, \MO_X) \neq 0$. 
We can find an effective $\Q$-divisor $\Delta$ 
such that $(X, \Delta)$ is klt and $-(K_X+\Delta)$ is ample. 
Fix a positive integer $\ell>0$ so that 
$\ell(K_X+\Delta)$ is Cartier and that 
$$H^i(X, \MO_X(-m\ell(K_X+\Delta)))=0$$ 
for $i>0$ and $m>0$. 
Set $\pi:Y \to X$ to be 
the $\mathbb A^1$-bundle over $X$, equipped with a section $E$, 
such that $\MO_Y(E)|_E \simeq \MO_X(\ell(K_X+\Delta))$ 
and that $Y$ is an open subset of the $\mathbb P^1$-bundle 
$\mathbb P_X(\MO_X \oplus \MO_X(\ell(K_X+\Delta)))$. 
Set $\Delta_Y:=\pi^*\Delta$. 
Then $E$ can be contracted, and we have a projective birational morphism 
$g:Y \to Z$ such that $\Ex(g)=E$ and $g(\Ex(g))$ is one point. 

Since $\big( (K_Y+\Delta_Y+E)-\frac{1}{\ell}E \big)|_E \sim_{\Q} 0$, we obtain 
$K_Y+\Delta_Y+ \big( 1-\frac{1}{\ell} \big)E \sim_{\Q} 0$.
Therefore, its push-forward $g_* \big( K_Y+\Delta_Y+(1-\frac{1}{\ell})E \big)=K_Z+\Delta_Z$ is $\Q$-Cartier, 
where $\Delta_Z:=g_*\Delta_Y$. 
It follows that 
\[
K_Y+\Delta_Y+ \big( 1-\frac{1}{\ell} \big) E=g^*(K_Z+\Delta_Z).
\]
Since there exists a log resolution of $(X, \Delta)$ and $Y$ is an $\mathbb A^1$-bundle over $X$, 
the pair $\big( Y, \Delta_Y+(1-\frac{1}{\ell})E \big)$ is klt. Therefore, $(Z, \Delta_Z)$ is also klt.

We show $R^1g_*\MO_Y \neq 0$. 
For this, it suffices to show that $R^ig_*\MO_Y(-E)=0$ for $i>0$ since we have $H^1(X, \MO_X) \neq 0$. 
We prove that $R^ig_*\MO_Y(-mE)=0$ for $m>0$ by the descending induction on $m$. 
For $m \gg 0$, we obtain $R^ig_*\MO_Y(-mE)=0$ by the Serre vanishing theorem. 
By an exact sequence: 
\[
0 \to \MO_Y(-(m+1)E) \to \MO_Y(-mE) \to \MO_Y(-mE)|_E \to 0,
\]
it is enough to show $H^i(E, \MO_Y(-mE)|_E)=0$ for $m>0$. 
This holds by 
\[
H^i(E, \MO_Y(-mE)|_E) \simeq H^i(X, \MO_X(-m\ell(K_X+\Delta)))=0.
\]
Therefore, we have $R^1g_*\MO_Y \neq 0$. 
\end{proof}

\begin{rmk}
With the same notation as in the proof of Proposition~\ref{ex:dim4}, 
we can show that 
$$R^ig_*(W\MO_{Y, \Q}) = 0$$
for any $i>0$. 
The sketch of the proof is as follows. 
Since $-E$ is $g$-ample, we can find $m \gg 0$ such that $R^ig_*\MO_Y(-mE)=0$ for $i>0$, which implies that  
$R^ig_*(WI_{E, \Q})=R^ig_*(WI_{mE, \Q})=0$ 
where $I_E$ and $I_{mE}$ are the coherent ideal sheaves of $E$ and $mE$ respectively 
(cf.\ \cite[Proposition 2.1]{BBE}). 
On the other hand, we can check that 
there exists a finite universal homeomorphism $X \to \mathbb P^2$ 
by the construction of \cite[Section~9]{schroer}. 
Thus we get $H^i(E, W\MO_{E, \Q})=H^i(X, W\MO_{X, \Q})=0$ (cf.\ Lemma \ref{geom-homeo}). 
Therefore, the required vanishing $R^ig_*(W\MO_{Y, \Q}) = 0$ follows from an exact sequence: 
$$0 \to WI_{E, \Q} \to W\MO_{Y, \Q} \to W\MO_{E, \Q} \to 0.$$
\end{rmk}

\section{Rational chain connectedness}\label{section:rcc}

In this section, we give a proof of the rational chain connectedness of 
threefolds of Fano type in characteristic $p>5$ even in the relative setting (cf.\ \cite[Theorem 1]{Shokurov}). 

\begin{thm}\label{thm:rcc_rel}
Let $k$ be a perfect field of characteristic $p>5$. 
Let $f:X \to Y$ be a projective morphism with $f_*\MO_X=\MO_Y$ 
from a klt threefold $(X, \Delta)$ such that $-(K_X+\Delta)$ is $f$-nef and $f$-big. 
Then every fiber of $f$ is rationally chain connected. 
\end{thm}

First we recall the definition of varieties of Fano type and the rational chain connectedness. 

\begin{defi}\label{defi:fanotype}
Let $X$ be a projective normal variety. 
We say that $X$ is \textit{of Fano type} if there exists an effective $\mathbb{Q}$-divisor 
$\Delta$ such that $-(K_X+\Delta)$ is ample and $(X, \Delta)$ is klt. 
\end{defi}

\begin{rmk}\label{rmk:nefbig}
If $\dim X \leq 3$ (cf.\ Remark~\ref{rmk:perturbation}) and there exists an effective $\mathbb{Q}$-divisor 
$\Delta$ such that $-(K_X+\Delta)$ is nef and big and $(X, \Delta)$ is klt, 
then $X$ is of Fano type (see \cite[Lemma-Definition 2.6]{PS}) thanks to the existence of log resolutions. 
\end{rmk}

\begin{rmk}\label{rmk:perturbation}
For a klt pair $(X, B)$ and an effective $\mbQ$-Cartier $\mbQ$-divisor $D$, 
there exists $\epsilon>0$ such that $(X, B+\epsilon D)$ is again klt, 
if there exists a log resolution of $(X, B+D)$. 
\end{rmk}

\begin{defi}[{\cite[IV. Definition 3.2]{Kollar2}}]
A scheme $X$ of finite type over a field $k$ is called \textit{rationally chain connected} 
if there exist a family of geometrically connected proper algebraic curves $g : U \to Y$ and a cycle morphism $u: U \to X$ 
with the following two conditions. 
\begin{itemize}
\item The geometric fibers of $g$ have only rational components. 
\item $u^{(2)}: U \times _Y U \to X \times _k X$ is dominant. 
\end{itemize}

\end{defi}
\begin{rmk}\label{rmk:rcc}
We list some basic properties on the rational chain connectedness. 
\begin{enumerate}
\item (\cite[IV. Exercise 3.2.5]{Kollar2} and \cite[Proposition A.4]{Tanaka:behavior}) 
For an arbitrary field extension $k'/k$, $X$ is rationally chain connected if and only if 
so is the base change $X \times_k k'$. 

\item (\cite[IV. Proposition 3.6]{Kollar2}) 
Suppose that $k$ is an uncountable algebraically closed field. 
Then, $X$ is rationally chain connected if and only if 
a pair of general closed points can be connected by a chain of rational curves. 
\end{enumerate}
\end{rmk}

\subsection{Birational and relative cases}
In this subsection, we prove Theorem~\ref{thm:rcc_rel} when $\dim Y > 0$. 
We start with a basic criterion for the rational chain connectedness.

\begin{lem}\label{RCC-lemma}
Let $f:X \to Y$ be a birational morphism of proper varieties over an algebraically closed field. 
Suppose that the following two conditions hold. 
\begin{enumerate}
\item{Every fiber of $f$ is rationally chain connected.}
\item{For a rational curve $C$ contained in $f(\Ex(f))$, 
its inverse image $f^{-1}(C)$ is rationally chain connected.}
\end{enumerate}
If a closed subscheme $\Gamma$ of $Y$ is rationally chain connected, 
then so is the inverse image $f^{-1}(\Gamma)$. 
\end{lem}

\begin{proof}
We can assume that $k$ is uncountable (Remark \ref{rmk:rcc} (1)). 
First, we reduce the problem to the case where $\Gamma$ is a rational curve. 
Let $x_1, x_2 \in f^{-1}(\Gamma)$. 
If $f(x_1)=f(x_2)$, then $x_1$ and $x_2$ can be connected 
by a chain of rational curves in $f^{-1}(\Gamma)$ by the condition (1). 
Suppose $f(x_1) \neq f(x_2)$. 
By the rationally chain connectedness of $\Gamma$, 
there exists a connected chain $\Gamma ' \subset \Gamma$ of rational curves which contains $f(x_1)$ and $f(x_2)$. 
Since the existence of a connected chain of rational curves which contains $x_1$ and $x_2$ follows from the assertion for $\Gamma'$, 
we may assume that $\Gamma$ is a connected chain of rational curves from the beginning. 
Moreover, by replacing $\Gamma$ with its irreducible component, 
we may assume that $\Gamma$ is a rational curve. 

Suppose that $\Gamma$ is a rational curve. 
If $\Gamma \subset f(\Ex(f))$, then the assertion follows from the condition (2). 
Suppose that $\Gamma \not\subset f(\Ex(f))$.
Set $\Gamma_X$ to be the proper transform of $\Gamma$ on $X$. 
Let $\Gamma \cap f(\Ex(f))=\{y_1, \cdots, y_r\}$. 
Since $f^{-1}(\Gamma)=\Gamma_X \cup \bigcup_{1\leq i\leq r} f^{-1}(y_i)$, 
it is rationally chain connected by the condition (1). 
\end{proof}

We show the birational case (cf.\ \cite[Corollary 3]{Shokurov}). 

\begin{thm}\label{birational-RCC}
Let $f:X \to Y$ be a proper birational morphism of normal threefolds 
over a perfect field of characteristic $p>5$. 
Suppose that $(Y, \Delta_Y)$ is klt for some effective $\mbQ$-divisor $\Delta_Y$. 
If a closed subscheme $\Gamma$ of $Y$ is rationally chain connected, 
then so is its inverse image $f^{-1}(\Gamma)$. 
\end{thm}

\begin{proof}
We may assume that the base field is an uncountable algebraically closed field. 

First, we shall reduce to the case where $Y$ is $\mbQ$-factorial. 
Let $y_1, \cdots, y_r$ be the non-$\Q$-factorial points on $Y$. 
By Proposition \ref{MFS-special}, we have a birational modification $g_Y:Y' \to Y$ 
such that $Y'$ is $\mathbb{Q}$-factorial and klt, $\Ex(g_Y)=\coprod_{1\leq i \leq r}g _Y ^{-1}(y_i)$, 
and each $T_i=g_Y^{-1}(y_i)_{\mathrm{red}}$ is a surface of Fano type. 
In particular, $T_i$ is rational, hence is rationally connected. 
Therefore, $g_Y^{-1}(\Gamma)$ is rationally chain connected. 
Further, there exist a normal threefold $X'$ and birational morphisms 
$f': X' \to Y'$ and $X' \to X$ such that the following diagram commutes: 
\[\xymatrix{
X' \ar[d]^{g_X} \ar[r]^{f'} & Y' \ar[d]^{g_Y} \\
X \ar[r]^{f}& Y
}\]
If we could prove that $f'^{-1}(g_Y^{-1}(\Gamma))$ is rationally chain connected, 
then so is its image $f^{-1}(\Gamma) = g_X \big( f'^{-1}(g_Y^{-1}(\Gamma)) \big)$ under $g_X$. 
Therefore, it is sufficient to prove the assertion for $f'$. 
Hence, by replacing $Y$ with $Y'$, we may assume that $Y$ is $\mbQ$-factorial. 

Suppose $Y$ is $\mbQ$-factorial. We may assume $\Delta _Y = 0$. 
Further, we may replace $f$ with a log resolution of $Y$. 
We run a $(K_X+E)$-MMP over $Y$, where $E$ is the reduced divisor with $\Supp E=\Ex(f)$. 
This MMP terminates and ends with $Y$ itself (see the proof of Theorem~\ref{main-normal}). 
Each step of this MMP is a pl-divisorial contraction or a pl-flip. 
Thus, it suffices to show the assertion under the assumption that $f:X \to Y$ is 
either a pl-divisorial contraction or pl-flipping contraction 
with respect to a dlt pair $(X, \Delta)$. 
The assertion follows from Lemma~\ref{RCC-lemma} and the argument in 
the proof in \cite[Proposition~3.6]{6authors}. 
\end{proof}

Using the previous proposition, we prove Theorem \ref{thm:rcc_rel} for $\dim Y >0$.

\begin{prop}\label{relative-RCC}
Let $f:X \to Y$ be a projective morphism of normal varieties 
over a perfect field of characteristic $p>5$ with $f_*\MO_X=\MO_Y$. 
Suppose that $\dim X \leq 3$ and that 
there exists an effective $\Q$-divisor $\Delta$ on $X$ such that $(X, \Delta)$ is klt 
and $-(K_X+\Delta)$ is $f$-nef and $f$-big. 
If $\dim Y>0$, then every fiber of $f$ is rationally chain connected. 
\end{prop}

\begin{proof}
We may assume that the base field is an uncountable algebraically closed field. 
Fix a closed point $z \in Y$. 
We shall show that $f^{-1}(z)$ is rationally chain connected.  
By Proposition~\ref{MFS-special}, there exists a commutative diagram 
$$\begin{CD}
W @>\varphi'>> X'\\
@VV\varphi V @VVf'V\\
X @>f>> Y,
\end{CD}$$
such that $\varphi$ and $\varphi'$ are projective birational morphisms, and 
$f'^{-1}(z)$ is rationally connected. 
By Theorem~\ref{birational-RCC}, $\varphi'^{-1}(f'^{-1}(z))$ is rationally chain connected, 
hence so is its image $f^{-1}(z) = \varphi \big( \varphi'^{-1}(f'^{-1}(z)) \big)$ under $\varphi$. 
\end{proof}

\begin{rmk}
If $f:X \to Y$ is a Mori fiber space, then 
it is known that a general fiber of $f$ is rationally chain connected by \cite[Theorem 0.4]{Tanaka:behavior} 
in arbitrary positive characteristic. 
\end{rmk}

\subsection{Global case}
In this subsection, we show the rational chain connectedness of log Fano threefolds. 

\begin{prop}\label{thm:rcc_Fanotype} 
Let $X$ be a projective geometrically connected threefold 
of Fano type over a perfect field of characteristic $p>5$. 
Then $X$ is rationally chain connected. 
\end{prop}

\begin{proof}
We may assume that the base field is an uncountable algebraically closed field. 
Taking a small $\mbQ$-factorialization (\cite[Theorem 1.6]{Birkar}) if necessary, 
we may assume that $X$ is $\mbQ$-factorial (cf.\ Remark \ref{rmk:nefbig}). 
Let $\Delta$ be an effective $\mathbb{Q}$-divisor 
$\Delta$ such that $-(K_X+\Delta)$ is ample and $(X, \Delta)$ is klt. 
Then for any divisor $B$, we can run a $B$-MMP. 
Indeed, we can take an ample $\mathbb{Q}$-divisor $A$ and positive integer $m$ 
such that $\frac{1}{m}B \sim_{\mathbb{Q}} K_X+\Delta+A$ and $(X,\Delta+A)$ is klt. 
Then a $(K_X + \Delta + A)$-MMP with scaling of some ample divisor runs and terminates by \cite[Theorem 1.6]{BW}. 
This MMP is also a $B$-MMP. 

Let $f: X \dashrightarrow S$ be a maximally rationally chain connected fibration as in \cite[IV. Definition 5.1]{Kollar2}, 
where $S$ is a projective normal variety. 
This is almost holomorphic, i.e.\ 
we can find non-empty open subsets $X' \subset X$ and $S' \subset S$ 
such that $f$ induces a proper morphism $X' \to S'$. Note that $\mathrm{dim}\,S<  \mathrm{dim}\,X$ by \cite[Proposition 3.4 (1)]{6authors}.
Let $H$ be an ample Cartier divisor on $S$ and 
let $D$ be the closure of $f^{*}H|_{X'}$. 
We consider $(-D)$-MMP:
\[
X=:X_0 \overset{\varphi_0}{\dashrightarrow}  X_1 \overset{\varphi_1}{\dashrightarrow} 
\cdots \overset{\varphi_{r-1}}{\dashrightarrow} X_r \dashrightarrow \cdots 
\]
Note that this MMP ends with a Mori fiber space. 
Let $D_i$ be the strict transform of $D$ on $X_i$. 
Let $r$ be the least $r$ satisfying the following condition: 
\begin{itemize} 
\item for the $(-D_r)$-negative extremal contraction $\psi _r : X_r \to W_{r}$ 
yielding $\varphi _r$, 
its exceptional locus $\Ex (\psi _r)$ is intersecting with a general fiber of $f_r:X_r \dashrightarrow S$, 
i.e.\ $\Ex (\psi _r)$ dominates $S$. 
\end{itemize}
Note that the induced rational map $f_i: X_i \dashrightarrow S$ is 
also a maximally rationally chain connected fibration for any 
$0\leq i\leq r$ and when $\psi _i$ is a Mori fiber space, we set $\Ex (\psi _i) :=X_i$ 
(hence such $r$ always exists). 

Suppose that $\dim S >0$ by contradiction. 
Fix very general $s \in S$. Then $s$ satisfies the following property 
(see \cite[IV. (5.2.1)]{Kollar2}): 
if $f_r^{-1}(s)$ intersects with a rational curve $C \subset X_r$ then $C \subset f_r^{-1}(s)$. 
By Proposition~\ref{relative-RCC} for $\mathrm{dim}\,W_r >0$ 
and \cite[Proposition 3.4 (2)]{6authors} for $\mathrm{dim}\,W_r =0$, 
some $\psi_r$-contracted rational curve $C$ intersects with $f_r ^{-1}(s)$, and
hence $C \subset f_r ^{-1}(s)$. 
Since $D_r$ is the pull-back of $H$ by $f_r$ around $C$, it follows that $C \cdot D_r=0$. 
However, this contradicts the $(-D_r)$-negativity of $\psi _r$. 
\end{proof}

\begin{rmk}
By the same proof as above, 
Proposition~\ref{thm:rcc_Fanotype} can be generalized to arbitrary dimension 
if we assume the existence of log resolution and the minimal model program. 
\end{rmk}

\begin{proof}[Proof of Theorem~\ref{thm:rcc_rel}]
We may assume that $-(K_X+\Delta)$ is $f$-ample. 
When $\dim Y > 0$, the assertion follows from Proposition~\ref{relative-RCC}. 
When $\dim Y = 0$, the assertion follows from Proposition~\ref{thm:rcc_Fanotype}. 
\end{proof}

\section{Rational points on log Fano contractions over finite fields}\label{section-rational-pt}
In this section, we discuss rational points on a (mainly Fano) variety over a finite field. 

First, we treat the birational case as follows:

\begin{thm}\label{rat-pt-birational}
Let $f:X \to Y$ be a birational morphism of normal threefolds over $\mathbb F_q$ of characteristics $p>5$. 
If there exist effective $\mbQ$-divisors $\Delta_X$ on $X$ and $\Delta_Y$ on $Y$ 
such that $(X, \Delta_X)$ and $(Y, \Delta_Y)$ are klt. 
Then 
\[
\# X(\mathbb F_q) \equiv \# Y(\mathbb F_q) \pmod q.
\]
\end{thm}

\begin{proof}
Taking a  resolution, we may assume that $X$ is smooth and that $f$ is a morphism. 
Since $Y$ has $W\MO$-rational singularities by Theorem \ref{thm:WO}, 
the assertion holds by Remark \ref{rmk:any} and \cite[Proposition 6.9]{BBE}. 
\end{proof}

We also have a global statement as follows. 

\begin{cor}
Let $X \dasharrow Y$ be a birational map of proper klt threefolds over $\mathbb F_q$ of characteristics $p>5$. 
Then $\# X(\mathbb F_q) \equiv \# Y(\mathbb F_q) \pmod q$. 
\end{cor}%
\begin{proof}
Taking a common resolution, the assertion follows from Theorem~\ref{rat-pt-birational}.
\end{proof}

Now, we prove one of the main theorems of this paper. 

\begin{thm}\label{thm:logfano_ratpo}
Let $X$ be a proper geometrically connected threefold over a finite field $\mbF _q$ 
of characteristic $p>5$. Suppose that $(X, \Delta)$ is klt for some effective $\mbQ$-divisor $\Delta$ 
and $X$ is rationally chain connected. 

Then $\# X(\mbF _q) \equiv 1 \pmod q$ holds. 
In particular, $X$ has a rational point. 
\end{thm}

\begin{proof}
Taking a resolution of singularities, 
we may assume that $X$ is smooth by Theorem~\ref{birational-RCC} and Theorem~\ref{rat-pt-birational}. 
Then the assertion follows from \cite[Theorem 1.1]{Esnault} and the argument between (1.1) and (1.4) in \cite{Esnault}.
\end{proof}

Next, we consider the relative case.

\begin{thm}\label{thm:rational-pt}
Let $f:X \to Y$ be a proper morphism of normal varieties over $\mathbb F_q$ of characteristics $p>5$ with $f_*\mcO_X=\mcO_Y$. 
Suppose that there exists an effective $\Q$-divisor $\Delta$ 
such that $(X, \Delta)$ is a klt threefold and $-(K_X+\Delta)$ is $f$-nef and $f$-big. 
Then 
\[
\# X_y (\mbF _q) \equiv 1 \pmod q
\]
holds for any $\mathbb F_q$-rational point $y$ on $Y$ and the fiber $X_y$ over $y$. 
In particular, 
$$\# X(\mbF _q) \equiv \# Y(\mathbb F_q)  \pmod q.$$
\end{thm}

\begin{proof}
We shall prove the first assertion, since the second assertion follows from the first assertion. 
We may assume that $\dim Y > 0$ by Theorem~\ref{thm:logfano_ratpo}. 
Further, we may assume that $f$ is projective by Chow's lemma and Theorem~\ref{rat-pt-birational}. 
Since the statement is local on $Y$, we may also assume that $Y$ has a unique $\mbF _q$-rational point $y$. 

By Proposition~\ref{MFS-special}, we have a commutative diagram 
 $$\begin{CD}
W @>\varphi'>> X'\\
@VV\varphi V @VVf'V\\
X @>f>> Y,
\end{CD}$$
such that $\varphi$ and $\varphi'$ are projective birational morphisms, and 
$S := f'^{-1}(y)$ is a surface of Fano type.  By Theorem~\ref{rat-pt-birational}, we have
\[
\# X(\mbF _q) \equiv \# X'(\mathbb F_q)  \pmod q.
\]
Further $\# X' (\mbF _q) = \# S (\mbF _q)$ holds as $y$ is the unique rational point on $Y$. 
Since $S$ is a surface of Fano type (hence a rational surface with rational singularities)
it follows that $\# S(\mbF _q) \equiv 1 \pmod q$, which shows
\[
\# X_y (\mbF _q) = \# X (\mbF _q) \equiv \# X'(\mathbb F_q) = \# S (\mbF _q) \equiv 1 \pmod q . 
\]
\end{proof}

By the same technique as above, we can also show the following vanishing theorem. 

\begin{thm}\label{thm:WV}
Let $X$ be a proper threefold over a perfect field of characteristic $p>5$. 
Suppose that $(X, \Delta)$ is klt for some effective $\mbQ$-divisor $\Delta$ 
and $X$ is rationally chain connected. 
Then $H^i (X, W \mcO _{X, \mbQ}) = 0$ for $i > 0$. 
\end{thm}

\begin{proof}
Let $\varphi:V \to X$ be a birational morphism from a smooth projective threefold $V$. 
Since $V$ is also rationally chain connected (Theorem~\ref{birational-RCC}), we obtain 
$H^i (V, W \mcO _V) \otimes_{\mbZ} \Q = 0$ for $i > 0$ by \cite[Theorem 1.1]{Esnault}. 
By Remark~\ref{cr-remark}, we have that $H^i (V, W \mcO _{V, \Q}) = 0$ for $i > 0$. 
Since $X$ has $W\MO$-rational singularities, 
it follows that $R^i \varphi _* (W\mcO _{V, \mbQ}) = 0$ for $i > 0$, 
which proves $H^i (X, W \mcO _{X, \mbQ}) = 0$ for $i > 0$
by the Leray spectral sequence. 
\end{proof}

Lastly, we conclude the main theorems using Theorem~\ref{thm:rcc_rel}. 

\begin{proof}[Proof of Theorem \ref{mainthm:ratpo}]
This follow from Theorem~\ref{thm:rcc_rel} and Theorem \ref{thm:logfano_ratpo}. 
\end{proof}
\begin{proof}[Proof of Theorem \ref{mainthm:WV}]
This follow from Theorem~\ref{thm:rcc_rel} and Theorem~\ref{thm:WV}. 
\end{proof}


\end{document}